\documentclass[12pt,reqno]{amsart}

\usepackage{graphicx,subfigure}
\usepackage{color}
\usepackage{hyperref}
\usepackage[english]{babel} 
\usepackage[utf8]{inputenc} 
\usepackage[T1]{fontenc} 
\usepackage{amsmath}
\usepackage{amsfonts}
\usepackage{amssymb}
\usepackage{amsthm}
\usepackage{comment}
\usepackage{lmodern}

\numberwithin{equation}{section} 
\numberwithin{table}{section} 

{ \theoremstyle{plain}
\newtheorem{theorem}{Theorem}[section]
\newtheorem{proposition}[theorem]{Proposition}
\newtheorem{lemma}[theorem]{Lemma}

\newtheorem{remark}[theorem]{Remark}
\newtheorem{definition}[theorem]{Definition}
}

\def\R{\mathbb{R}}

\def\C{\mathbb{C}}
\def\N{\mathbb{N}}

\def\eps{\varepsilon}
\def\dx{\,\mathrm{d}x}
\def\dsigma{\,\mathrm{d}\sigma}
\def\dy{\,\mathrm{d}y}
\def\drho{\,\mathrm{d}\rho}

\begin{document}

\title[]{On the eigenvalues of Aharonov--Bohm operators with varying poles: pole approaching the boundary of the domain}

\author{Benedetta Noris}
\address{Benedetta Noris \newline \indent INdAM-COFUND Marie Curie Fellow,
\newline \indent Dipartimento di Matematica e Applicazioni, Universit\`{a} degli Studi di Milano-Bicocca,
\newline \indent via Bicocca degli Arcimboldi 8, 20126 Milano, Italy.}
\email{benedettanoris@gmail.com}

\author{Manon Nys}
\address{Manon Nys \newline \indent Fonds National de la Recherche Scientifique- FNRS.
\newline \indent Département de Mathématiques, Université Libre de Bruxelles (ULB), 
\newline \indent Boulevard du triomphe, B-1050 Bruxelles, Belgium.
\newline \indent Dipartimento di Matematica e Applicazioni,  Universit\`{a} degli Studi di Milano-Bicocca,
\newline \indent via Bicocca degli Arcimboldi 8, 20126 Milano, Italy.}
\email{manonys@gmail.com}

\author{Susanna Terracini }
\address{Susanna Terracini \newline \indent Dipartimento di Matematica “Giuseppe Peano”, Universit\`{a} di Torino,
\newline \indent Via Carlo Alberto 10, 20123 Torino, Italy.}
\email{susanna.terracini@unito.it}

\thanks{The authors are partially supported by the project ERC
Advanced Grant  2013 n. 339958: ``Complex Patterns for Strongly Interacting Dynamical Systems -
COMPAT''. M. Nys is a Research Fellow of the Belgian Fonds de la Recherche Scientifique - FNRS. The authors wish to thank Virginie Bonnaillie-No\"el for providing Figures \ref{fig.pole_approach_boundary} and \ref{fig.lambda_2and3_a}}

\date{\today}

\maketitle

\begin{abstract}
We continue the analysis started in \cite{BonnaillieNorisNysTerracini2013,TN1}, concerning the behavior of the eigenvalues of a magnetic Schr\"odinger operator of Aharonov-Bohm type with half-integer circulation.
We consider a planar domain with Dirichlet boundary conditions and we concentrate on the case when the singular pole approaches the boundary of the domain.
The $k$-th magnetic eigenvalue converges to the $k$-th eigenvalue of the Laplacian. We can predict both the rate of convergence and whether the convergence happens from above or from below, in relation with the number of nodal lines of the $k$-th eigenfunction of the Laplacian.
The proof relies on a detailed asymptotic analysis of the eigenfunctions, based on an Almgren-type frequency formula for magnetic eigenfunctions and on the blow-up technique.
%Primary ; secondary .
\end{abstract}

\section{Introduction}
In this paper we continue the analysis started in \cite{BonnaillieNorisNysTerracini2013,TN1}, concerning the behavior of the eigenvalues of the magnetic Schr\"odinger operator
\begin{equation}\label{eq:magnetic_operator_def}
(i \nabla + A_{a})^{2}=-\Delta +i\nabla\cdot A_{a} +2i A_{a}\cdot \nabla +|A_{a}|^2
\end{equation}
as the pole $a\in\Omega$ moves inside of $\Omega$ and eventually hits the boundary $\partial \Omega$. Here $\Omega\subset\R^2$ is open, bounded and simply connected, and we impose zero boundary conditions on $\partial \Omega$. The magnetic potential $A_a$ has the form
\begin{equation}\label{eq:magnetic_potential_definition}
A_{a}(x) = \frac{1}{2} \left( - \frac{x_{2} - a_{2}}{(x_{1} - a_{1})^{2}+(x_{2} - a_{2})^{2}}, \frac{x_{1} - a_{1}}{(x_{1} - a_{1})^{2}+(x_{2} - a_{2})^{2}} \right),
\end{equation}
for $a=(a_1,a_2)\in\Omega$ and $x=(x_{1},x_{2})\in\Omega\setminus\{a\}$.
Such magnetic potential is generated by a very thin solenoid orthogonal to the plane, and the associated magnetic field is a $\pi$-multiple of the Dirac delta orthogonal to the plane. Though the magnetic field vanishes almost everywhere, its presence affects the spectrum of the operator, giving rise to the so called Aharonov-Bohm effect \cite{AB}. 

With no additional effort with respect to the standard eigenvalue problem, we can introduce a weight
\begin{equation}\label{eq:weight_assumptions}
p \in C^{\infty}(\overline{\Omega}), \quad p(x)>0, \ x\in\Omega,
\end{equation}
and study the following modified eigenvalue problem
\begin{align}\label{eq:main_with_weight}
(i \nabla + A_a)^2 \varphi_k^a = \lambda_k^a \, p(x) \varphi_k^a, \quad \text{ with } \quad \varphi_k^a =0 \text{ on } \partial\Omega.
\end{align}
Throughout the paper, we will denote by $\lambda_k^a$, $k\in \N=\{1,2,\ldots\}$, the eigenvalues in \eqref{eq:main_with_weight} arranged in an increasing sequence and counted with their multiplicity. We will reserve the notation $\lambda_k$, $\varphi_k$ for the eigenvalues and eigenfunctions of the Laplacian with weight $p(x)$ in $\Omega$ and with zero boundary conditions (again increasing and counted with their multiplicity). Of course, our main interest is the case $p(x)=1$.

The study of the behavior of the eigenvalues and the link with the nodal domains of the corresponding eigenfunctions is motivated by its relation to spectral minimal partitions, see \cite{BonnaillieLena2012,HelfferSurvey,HH2012,HH2013,HHT2009,HHT2010dim3,HHT2010}. It has been proved in \cite{HHT2009} that, in dimension 2, if all the clustering points of a minimal partition have an even multiplicity, then the partition is nodal. The results in \cite{BH,BHH,BHV,TN1} establish a strong relation between  the minimal partitions having points with odd multiplicity with the nodal domains of Aharonov-Bohm eigenfunctions.

The main results obtained in \cite{BonnaillieNorisNysTerracini2013} and \cite{TN1} in connection with this topic concern the critical points of the map
\[
a\in\Omega\mapsto \lambda_k^a.
\]
In particular, in  \cite{BonnaillieNorisNysTerracini2013} the authors have analyzed the relation between such critical points and the nodal properties of the corresponding eigenfunctions. To this aim, let us recall the definition of order of vanishing of a function $u$ at an interior point $b$.
\begin{definition}[Interior zero of order $h/2$]
Let $u:\Omega\subset\R^2\to\C$, $b\in\Omega$ and $h\in\N$.
\begin{itemize}
\item[(i)] If $h$ is even, we say that $u$ has a zero of order $h/2$ at $b$ if it is of class at least $C^{h/2}$ in a neighborhood of $b$ and $u(b)=\ldots=D^{h/2-1}u(b)=0$, while $D^{h/2}u(b)\neq0$.
\item[(ii)] If $h$ is odd, we say that $u$ has a zero of order $h/2$ at $b$ if $u(x^2+b)$ has a zero of order $h$ at $x=0$ (here $x^2$ is the complex square).
\end{itemize}
\end{definition}

Then, a link between the order of vanishing of the function $a\mapsto \lambda_k^a$ can be established: indeed, the following result holds in case $p(x)\equiv1$ (but it is valid for any $p$ satisfying \eqref{eq:weight_assumptions}).

\begin{theorem}[{\cite{BonnaillieNorisNysTerracini2013,TN1}}] \label{theorem:comportmentofeigenvelue}
Let $\Omega\subset\R^2$ open, bounded and simply connected. Let $b\in\Omega$ and $\lambda_k^{b}$ simple. $\varphi_k^{b}$ has a zero of order $h/2$ at $b$, with $h\geq3$ odd, if and only if $b$ is an extremal point of the map $a\mapsto \lambda_k^a$. Moreover, in this case we have
\begin{align}\label{eq:expansion}
|\lambda_k^{a} - \lambda_k^{b}|  \leq C |a-b|^{(h+1)/2} \quad \text{ as }a\to b,
\end{align}
for a constant $C>0$ independent of $a$.
\end{theorem}

In this paper we address the study of the behavior of the magnetic eigenvalues as the pole approaches the boundary of the domain, starting from the observation that, since simply connectedness is restored when the pole lays on the boundary, there holds
\begin{theorem}[{\cite{BonnaillieNorisNysTerracini2013,Lena}}]\label{theorem:continuity}
Let $\Omega\subset\R^2$ open, bounded and simply connected.
For every $k\in \N$ we have that $\lambda_k^{a}\to \lambda_k$ as $a\to\partial\Omega$.
\end{theorem}

Our first aim is to obtain an expansion similar to \eqref{eq:expansion}  when the point $b$ is on the boundary $\partial\Omega$. It is worthwhile noticing that the techniques that were developed in \cite{BonnaillieNorisNysTerracini2013}, based on the application of local inversion methods, hardly extend to the study of the behavior at the boundary. For this reason, we propose here a different and more efficient approach, based on the min-max characterization of the eigenvalues.

Another phenomenon enlightened by the numerical simulations in \cite{BonnaillieNorisNysTerracini2013} attracted our attention. The convergence $\lambda_k^{a}\to \lambda_k$ described in Theorem \ref{theorem:continuity} can take place either from above or from below, depending on the value of $k$ and on the position of the pole, see Figure \ref{fig:lambda_j_a}. Of course, by the diamagnetic inequality, $\lambda_1^a>\lambda_1$ for every $a\in\Omega$. 
A more detailed analysis suggests that the different behaviors are related to the position of the pole with respect to the nodal lines of $\varphi_k$. If the pole $a$ moves from $\partial\Omega$ along a nodal line of $\varphi_k$, then the nodal line of $\varphi_k^a$ is shorter than the one of $\varphi_k$, locally for $a$ close to the boundary, see Figures \ref{fig:phi_3} and \ref{fig:phi_3a}. This determines a decrease of the energy: we can observe in Figure \ref{fig:lambda_3_a} that $\lambda_3^a<\lambda_3$ for $a$ approaching the point $(1,0)\in\partial\Omega$ along the symmetry axis of the circular sector. Conversely, if $a$ moves from $\partial\Omega$ not on a nodal line of $\varphi_k$ as in Figures \ref{fig:phi_2}-\ref{fig:phi_2a}, this creates a new nodal line in the magnetic eigenfunction and consequently an increase of the energy.

\begin{figure}[h!t]
\begin{center}
\subfigure[$\varphi_2$ \label{fig:phi_2}]{\includegraphics[height=3cm]{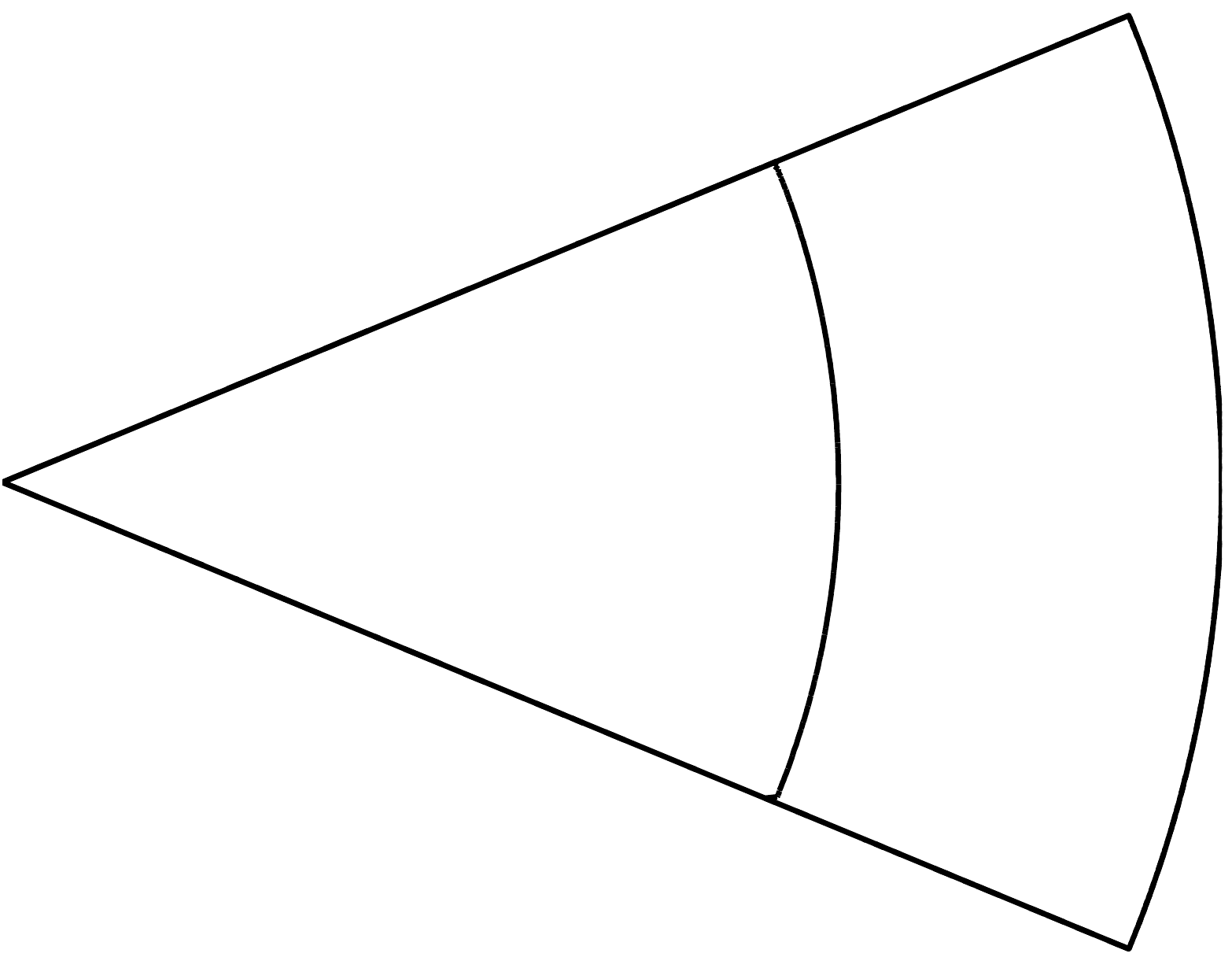}}
\subfigure[$\varphi_2^a$, $a=(0.9,0)$ \label{fig:phi_2a}]{\includegraphics[height=3cm]{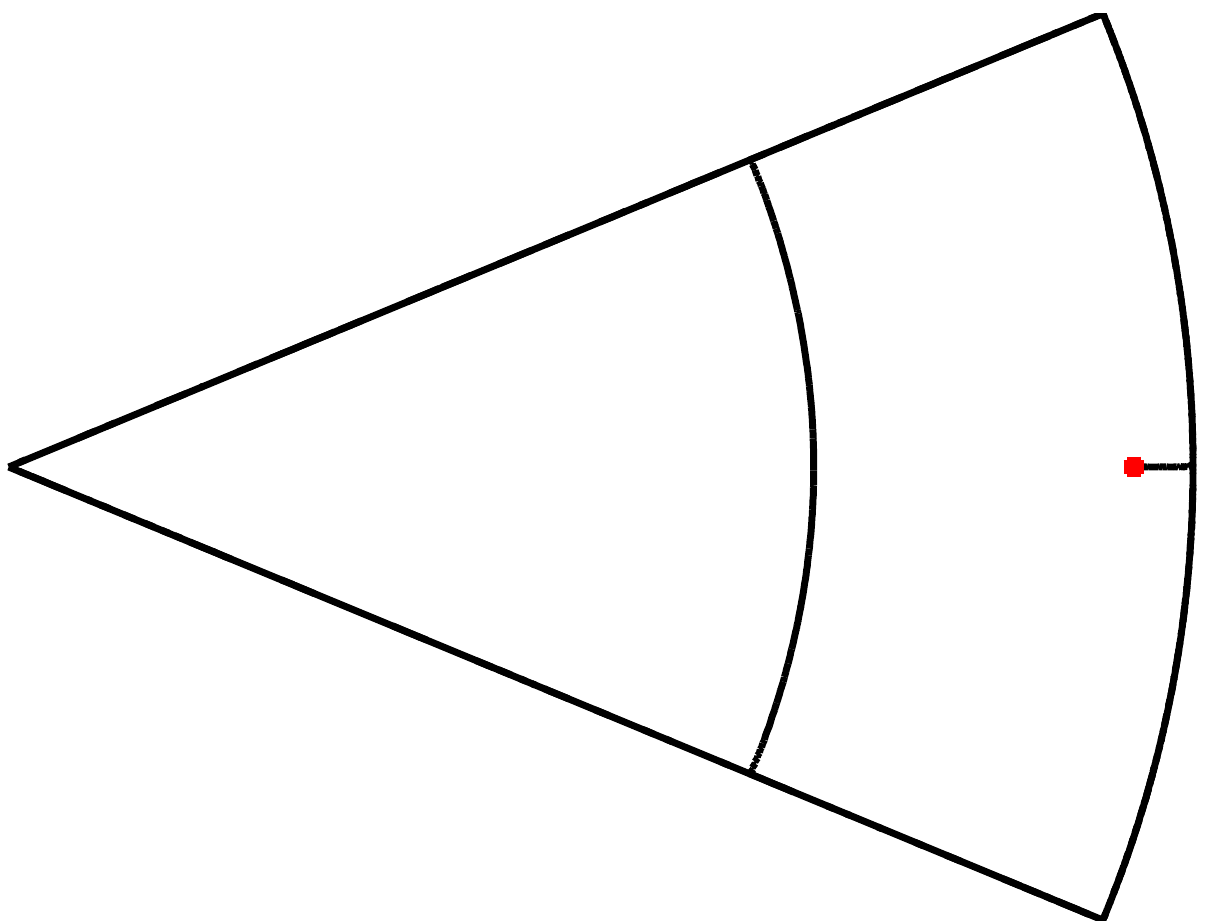}}
\subfigure[$\varphi_3$ \label{fig:phi_3}]{\includegraphics[height=3cm]{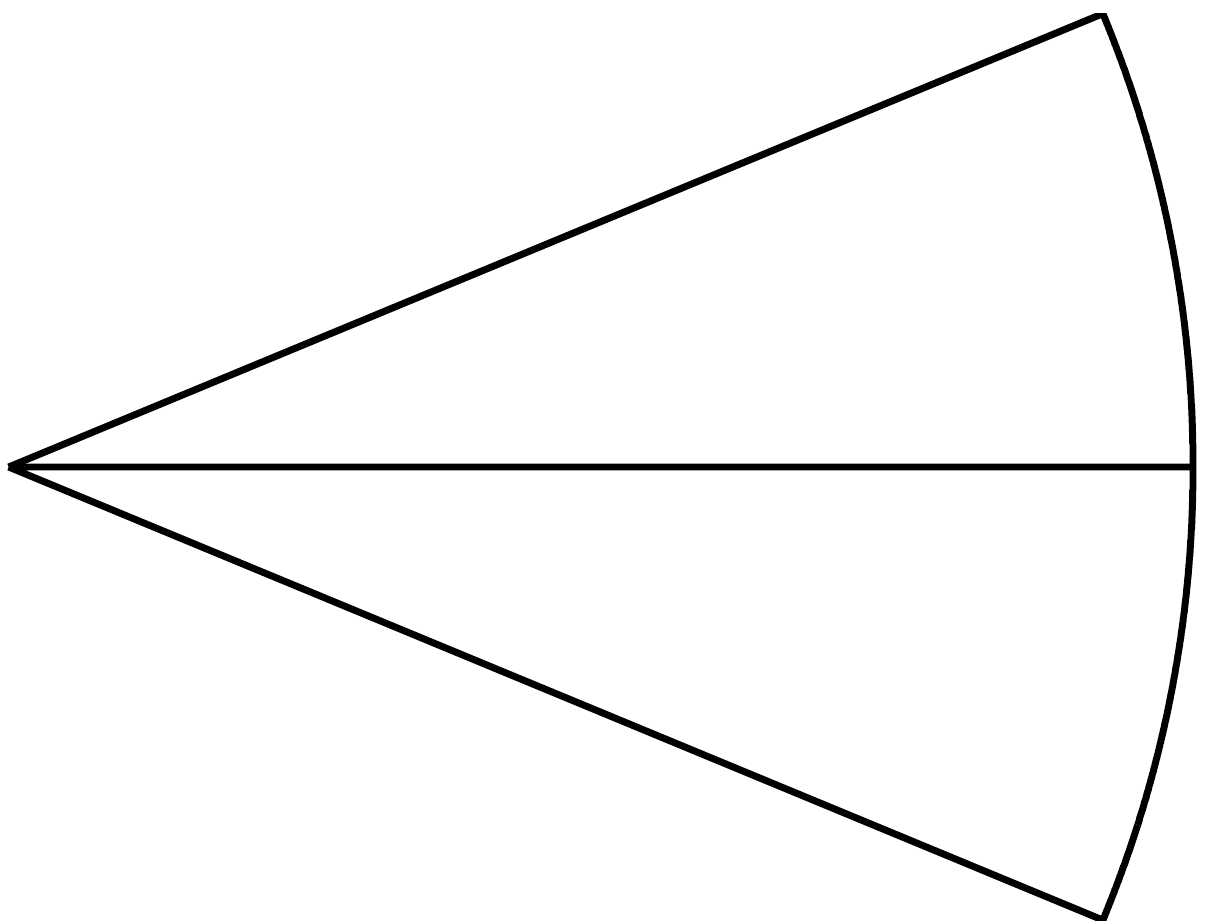}}
\subfigure[$\varphi_3^a$, $a=(0.9,0)$ \label{fig:phi_3a}]{\includegraphics[height=3cm]{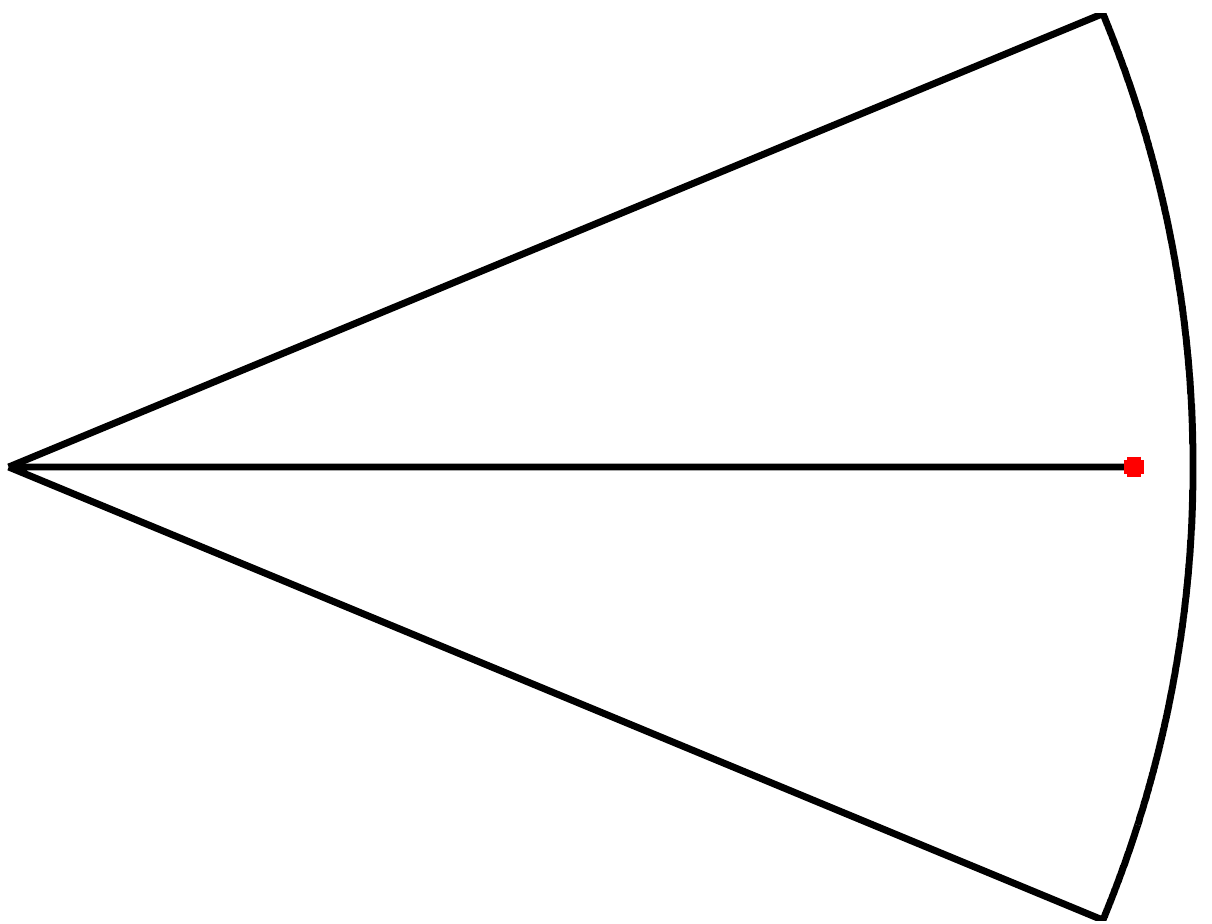}}
\caption{Nodal lines of second and third eigenfunctions in the circular sector of aperture $\pi/4$.\label{fig.pole_approach_boundary}}
\end{center}
\end{figure}

\begin{figure}[h!t]
\begin{center}
\subfigure[$a\mapsto \lambda_j^a$, $j=1,\ldots,9$, $a$ belonging to the symmetry axis \label{fig:lambda_j_a}]{\includegraphics[height=3.8cm]{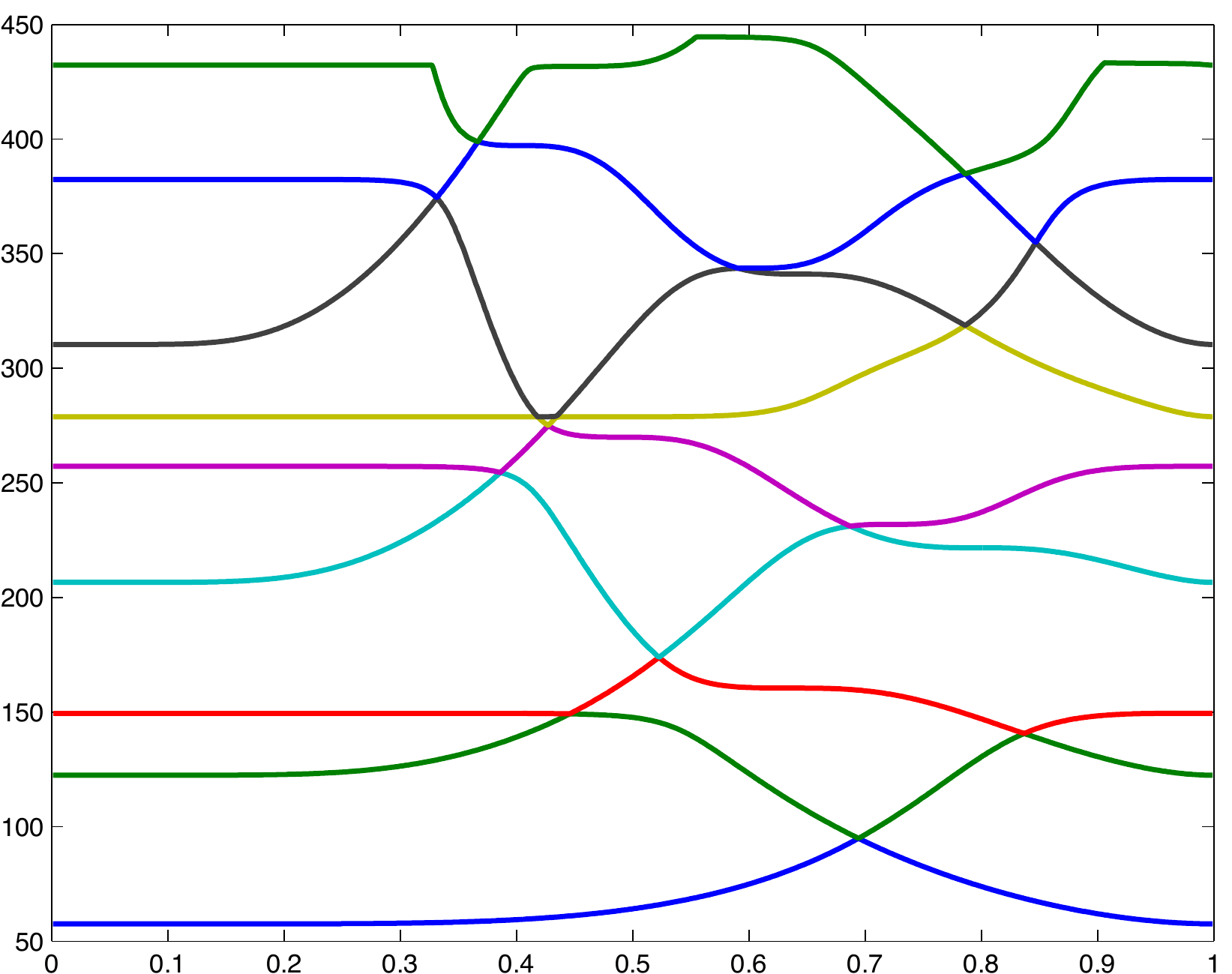}}
\hspace*{0.2cm}
\subfigure[$a\mapsto \lambda_2^a$\label{fig:lambda_2_a}]{\includegraphics[height=3.8cm]{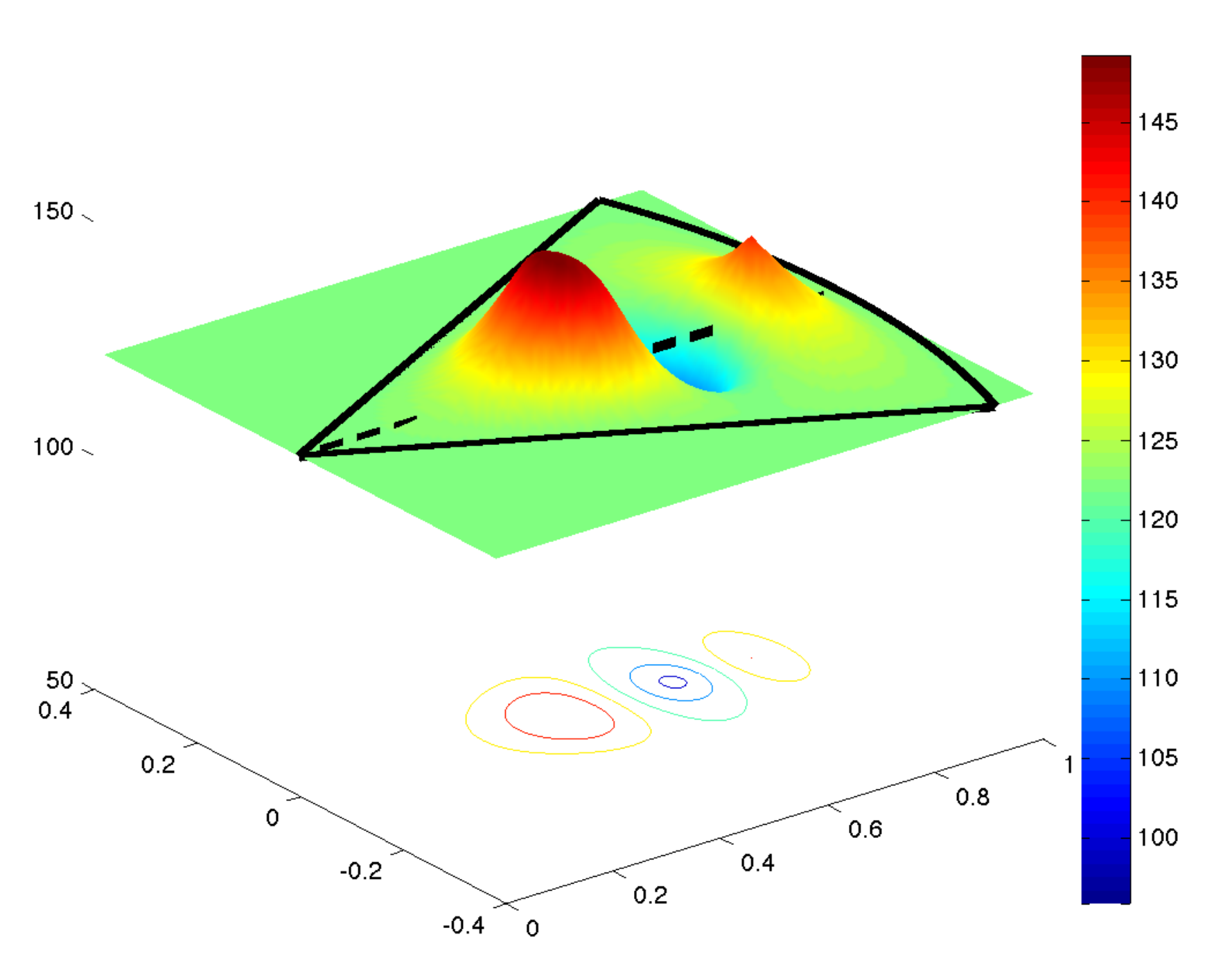}}
\hspace*{0.2cm}
\subfigure[$a\mapsto \lambda_3^a$\label{fig:lambda_3_a}]{\includegraphics[height=3.8cm]{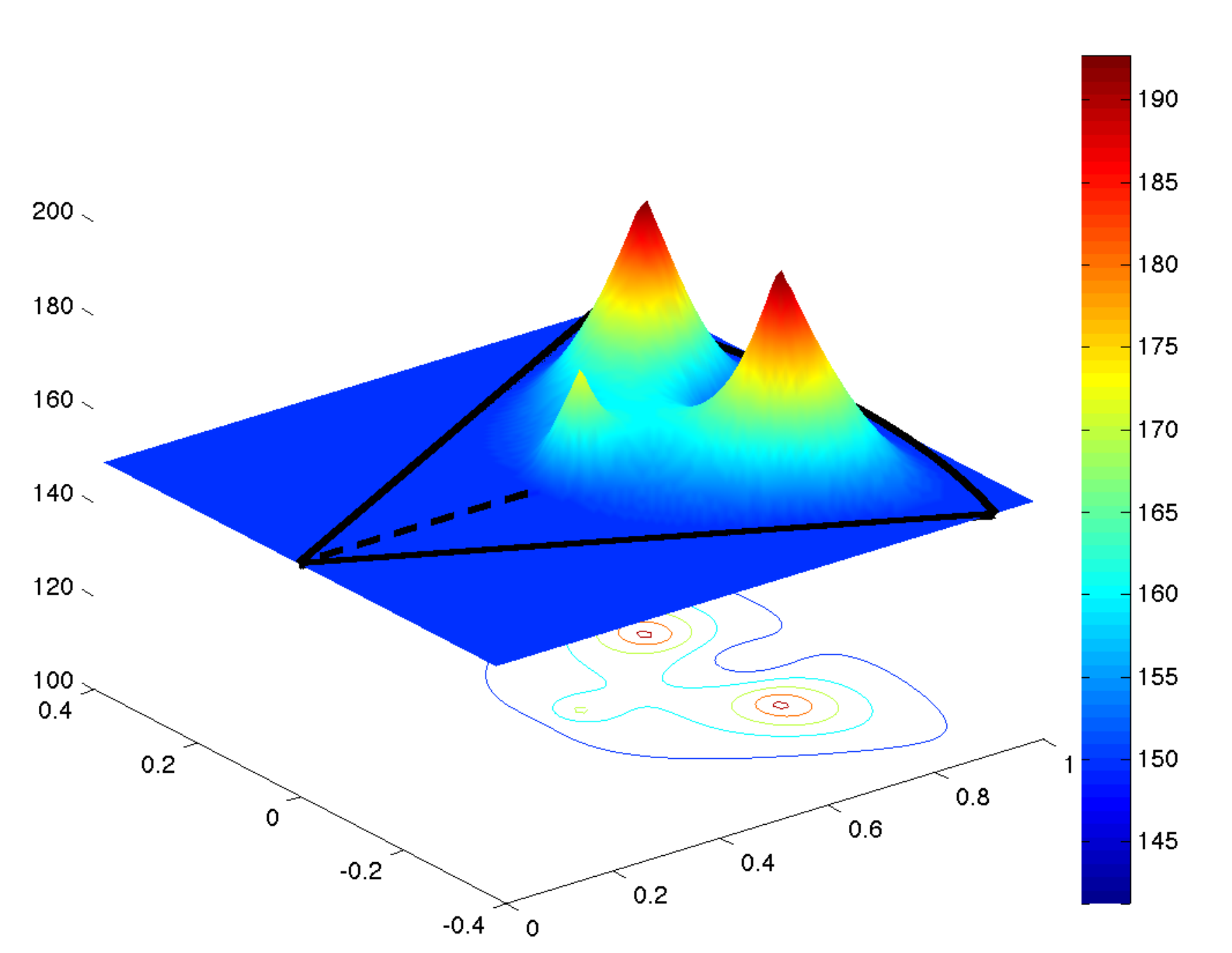}}
\caption{The map $a\mapsto \lambda_j^a$ for the circular sector of aperture $\pi/4$.\label{fig.lambda_2and3_a}}
\end{center}
\end{figure}

The second aim of this paper is to provide a theoretical justification of these facts. Before stating our main results, let us give the definition of order of vanishing at a boundary point $b\in\partial\Omega$ of a function $u$, with $u=0$ on $\partial\Omega$. As we will see, this definition makes sense only if $\partial\Omega$ is sufficiently regular (see the case of conical singularities in Appendix \ref{appendix:cone}).
\begin{definition}[Boundary zero of order $h/2$]
Let $\Omega\subset\R^2$ open, bounded and of class $C^\infty$.
Let $u:\Omega\to\C$, $u=0$ on $\partial\Omega$, $b\in\partial \Omega$ and $h$ even.

We say that $u$ has a zero of order $h/2$ at $b$ if there exists a neighborhood $U(b)$ such that $u\in C^{h/2}(U(b)\cap\Omega)$ and $u(b)=\ldots=D^{h/2-1}u(b)=0$ while $D^{h/2}u(b)\neq0$ in $U(b)\cap\Omega$.
\end{definition}
Note that, whereas for $b\in\Omega$ a zero of order $h/2$ corresponds to $h$ arcs of nodal lines meeting at $b$, for $b\in\partial\Omega$ a zero of order $h/2$ corresponds to $h/2-1$ arcs of nodal lines meeting at $b$. This is due to the fact that we are considering zero boundary conditions. Our first main result is the following.

\begin{theorem} \label{theorem:at_least_one_nodal_line}
Let $\Omega\subset\R^2$ open, bounded, simply connected and of class $C^\infty$. Let $p$ satisfy \eqref{eq:weight_assumptions}.
Suppose that $\lambda_{k-1}<\lambda_{k}$ and that there exists an eigenfunction $\varphi_{k}$ associated to $\lambda_k$ having a zero of order $h/2\geq 2$ at $b\in\partial\Omega$, i.e. at least one piece of nodal line ending at $b$. Denote by $\Gamma$ any such piece of nodal line. Then there exists $C > 0$, not depending on $a$, such that
\begin{align} \label{eq:at_least_one_nodal_line}
\lambda_{k}^{a} \leq \lambda_{k} - C |a-b|^{h } \quad \text{for } a\in\Gamma,\ a\to b.
\end{align}
\end{theorem}

The study of the case when the pole approaches $\partial \Omega$ at a point where no nodal lines of $\varphi_k$ end, requires additional work. The difficulty is that, in order to prove the opposite inequality with respect to \eqref{eq:at_least_one_nodal_line}, we need some information about the behavior of $\varphi_k^a$ when $a$ is close to the boundary. In this direction, we prove the uniqueness of the following limit profile.

\begin{proposition}\label{prop:unique_limit_profile}
Let $e=(1,0)$ and let $\psi$ be a solution to
\begin{align} \label{eq:limit_profile}
\left\{ \begin{array}{ll} (i \nabla + A_e)^2 \psi = 0 \quad & \mathbb{R}^{2}_+ \\
                            \psi  = 0 \quad & \{x_1=0\},
                            \end{array} \right. 
\end{align}
satisfying the normalization condition
\begin{equation}\label{eq:limit_profile_normalization}
\lim_{r\to+\infty} \frac{r\| (i\nabla+A_e)\psi \|^2_{L^2(D_r^+(0))}}{\|\psi\|^2_{L^2(\partial D_r^+(0))}}  =1,
\end{equation}
where $D_r^+(0):=D_r(0)\cap\{x_1>0\}$. Then
\begin{itemize}
\item[(i)] $\psi$ is unique up to a multiplicative constant;
\item[(ii)] for $r>1$ we have $\displaystyle \psi(r,\theta)=C e^{i\theta_e/2}\left(r\cos\theta -\frac{\beta}{\pi}\frac{\cos\theta}{r}+ \sum_{n\geq3, n \, odd} \frac{b_n}{r^n} \cos(n\theta) \right)$, where $\beta>0$ is explicitly characterized in \eqref{eq:beta_def} and $b_n\in\R$.
\end{itemize}
\end{proposition}

Once the uniqueness is established, we can prove that this limit profile provides a good description of the asymptotic behavior of $\varphi_k^a$. In the following theorem, we locally flatten the boundary of $\Omega$ near the boundary point through a conformal transformation (see Section \ref{sec:rectified_boundary} for more details), which allows to work on half-balls. Then we perform a normalized blow up, which converges to the previous profile on the half-space in the limit.

\begin{theorem} \label{theorem:blow_up_almgren}
Let $\Omega\subset\R^2$ open, bounded, simply connected and of class $C^\infty$. Let $p$ satisfy \eqref{eq:weight_assumptions}. 
Suppose that $\varphi_{k}$ has a zero of order $1$ at $b\in\partial\Omega$ (no nodal lines ending at $b$). 

Let $\Phi$ be a conformal map such that $\Phi^{-1}\in C^\infty(\overline{\Omega})$, $\Phi^{-1}(b)=0$ and, for some small $r>0$,
\[
\Phi^{-1}(\Omega)\cap D_r(0)=\{ x\in D_r(0): \, x_1>0 \}=:D_r^+(0).
\]
Then there exists $K>1$ such that, denoting by
\begin{equation*}
a'=(a_1',a_2')=\Phi^{-1}(a), \qquad
\psi_k^a(y)=\frac{\sqrt{Ka'_1}}{\|\varphi_k^a\circ \Phi\|_{L^2(\partial D^+_{Ka'_1}(0,a'_2))} } \varphi_k^a(\Phi(a'_1 y_1, a'_1 y_2 + a'_2))
\end{equation*}
we have
\[
\psi_k^a\to\psi \quad\text{in } H^1_{A_{e},loc}(\R^2_+) \text{ as } a\to b,
\]
where $e = (1,0)$ and $\psi$ is the unique solution to \eqref{eq:limit_profile}-\eqref{eq:limit_profile_normalization} with multiplicative constant $C$ given explicitely in \eqref{eq:constante_normalisation}.
\end{theorem}

The previous two results are obtained by exploiting an Almgren-type frequency formula \cite{AlmgrenBook,Almgren} for magnetic eigenfunctions, see Definition \ref{def:EHN}. This tool has been introduced in the context of magnetic operators in \cite{FFT2011} to obtain, among other results, sharp regularity results for Aharonov-Bohm eigenfunctions. 

The asymptotic analysis above allows us to prove the last main result of the paper.

\begin{theorem} \label{theorem:no_nodal_lines}
Let $\Omega\subset\R^2$ open, bounded, simply connected and of class $C^\infty$. Let $p$ satisfy \eqref{eq:weight_assumptions}.
Suppose that $\lambda_{k}$ is simple and that $\varphi_{k}$ has a zero of order $1$ at $b\in\partial\Omega$ (no nodal lines ending at $b$). Then there exists $C > 0$, not depending on $a$, such that
\begin{align}\label{eq:no_nodal_lines}
\lambda_{k}^a \geq \lambda_{k} + C \left(\text{dist}(a,\partial \Omega)\right)^2 \quad \text{as } a\to b.
\end{align}
\end{theorem}

\begin{remark}\label{rem:three_remarks}
(i) The behavior as $a\to b\in\partial\Omega$, $b$ being the endpoint of one or more nodal lines of $\varphi_k$, but $a$ not belonging to any such nodal line, remains an open problem. \\
(ii) The assumption of regularity of $\Omega$ can be weakened: it is enough to have $\partial\Omega\in C^{2,\gamma}$ for some $\gamma>0$, see Remark \ref{rem:regularity_Omega}. \\
(iii) If $\Omega$ presents a conical singularity, estimates \eqref{eq:at_least_one_nodal_line} and \eqref{eq:no_nodal_lines} do not hold at the vertex of the cone. This can be observed in the numerical simulations: we see in Figure \ref{fig.lambda_2and3_a} that the curve $a\mapsto\lambda_k^a$ is flat as $a$ approaches the acute angle of the circular sector. Concerning for example $\lambda_2^a$, we see that relation \eqref{eq:no_nodal_lines} does not hold, despite the absence of nodal lines of $\varphi_2$. We treat this topic in Appendix \ref{appendix:cone}. If we particularize Theorem \ref{theorem:cone_at_least_one_nodal_line} to the case of the second eigenvalue in the cone of aperture $\pi/4$, we obtain for example
\[
\lambda_2^a \geq \lambda_2+C|a|^8,
\]
as $a$ moves along the angle bisector.
\end{remark}

The paper is organized as follows. 
In Section \ref{sec:preliminaries} we define two function spaces related with the magnetic operator, $\mathcal{D}^{1,2}_{A_a}(\Omega)$ and $H^{1}_{A_a}(\Omega)$, respectively of functions with and without zero boundary conditions on $\partial \Omega$. We also recall a Hardy-type inequality and the asymptotic behavior of the eigenfunctions $\varphi_k$ and $\varphi_k^a$ around zeros of order $h/2$. 
In Section \ref{sec:rectified_boundary} we use a Riemann mapping theorem in  order to flatten locally the boundary $\partial \Omega$ around the point $0$. By doing this, we recover an equation similar to \eqref{eq:main_with_weight}, the new weight verifying the same assumptions as the old one, see \eqref{eq:weight_assumptions}, thanks to the regularity hypothesis on $\partial \Omega$. Then, we prove some Poincaré-type inequalities in half balls. 
In Section \ref{sec:pole_approaching_nodal} we prove Theorem \ref{theorem:at_least_one_nodal_line}. 
In Section \ref{sec:freq_formula} we introduce the Almgren's function for the eigenfunctions $\varphi_k^a$ and we study some properties of this object. 
In Section \ref{sec:prop_limit_profile} we prove Proposition \ref{prop:unique_limit_profile} and
in Section \ref{sec:pole_approaching_not_nodal} we prove Theorem \ref{theorem:no_nodal_lines}. 
In Appendix \ref{appendix:cone} we treat the case of a Lipschitz domain with isolated conical singularities.
Finally, in Appendix \ref{appendx:green} we prove an estimate on the Green's function for a perturbed Laplace operator.

\section{Preliminaries}\label{sec:preliminaries}

We introduce two function spaces related with the magnetic operator, one with vanishing conditions on $\partial \Omega$ and the other without any boundary conditions. We will denote them respectively by $\mathcal{D}^{1,2}_{A_a}(\Omega)$ and $H^{1}_{A_a}(\Omega)$. The first space $\mathcal{D}^{1,2}_{A_a}(\Omega)$ is defined as the completion of $C^{\infty}_{0}(\Omega \backslash \{a\},\C)$ with respect to the norm
\[
\| u \|_{\mathcal{D}^{1,2}_{A_a}(\Omega)} := \| (i \nabla + A_{a}) u \|_{L^{2}(\Omega)},
\]
while the second space $H^1_{A_a}(\Omega)$ is defined as the completion of the set
\begin{align*}
\left\{ u  \in C^{\infty}(\Omega,\C) \, , \, u \text{ vanishes in a neighborhood of } a \right\}
\end{align*}
with respect to the norm
\begin{align*}
\| u \|_{H^1_{A_a}(\Omega)} := \left( \| (i \nabla + A_a) u \|_{L^2(\Omega)}^2 + \| u \|_{L^2(\Omega)}^2 \right)^{1/2}.
\end{align*}
As proved for example in \cite[Lemma 2.1]{TN1}, we have the equivalent characterizations
\begin{align*} 
\mathcal{D}^{1,2}_{A_a}(\Omega)  = \left\{ u \in H^{1}_{0}(\Omega)  \, : \, \frac{u}{|x-a|} \in L^{2}(\Omega) \right\}
\end{align*}
and
\begin{align*}
H^{1}_{A_a}(\Omega)  = \left\{ u \in H^{1}(\Omega)  \, : \, \frac{u}{|x-a|} \in L^{2}(\Omega) \right\},
\end{align*}
and moreover we have that $\mathcal{D}^{1,2}_{A_a}(\Omega)$ (respectively $H^1_{A_a}(\Omega)$) is continuously embedded in $H^1_0(\Omega)$ (respectively $H^1(\Omega)$) : there exists a constant $C>0$ such that for every $u \in \mathcal{D}^{1,2}_{A_a}(\Omega)$ and $u \in H^1_{A_a}(\Omega)$ we have
\begin{equation}\label{eq:H_1_A_a_characterization}
\|u\|_{H_0^1(\Omega)}\leq C \|u\|_{\mathcal{D}^{1,2}_{A_a}(\Omega) } \quad \text{ and } \quad \|u\|_{H^1(\Omega)}\leq C \|u\|_{H^{1}_{A_a}(\Omega) }.
\end{equation}
This is proved by making use of a Hardy-type inequality, which was obtained by Laptev and Weidl \cite{LaWe} for functions in $\mathcal{D}^{1,2}_{A_a}(\Omega)$ and has been extended to functions in $H^{1}_{A_a}(\Omega)$ in \cite[Lemma 7.4]{MOR} (see also \cite{MORerratum}). If $\Omega$ is simply connected and of class $C^\infty$, there exists a constant $C>0$ such that for every $u \in H^1_{A_a}(\Omega)$ the following holds
\begin{equation}\label{eq:hardy}
\left\|\frac{u}{|x-a|}\right\|_{L^2(\Omega)} \leq C 
\| (i \nabla + A_{a}) u \|_{L^{2}(\Omega)}.
\end{equation}
The constant $C$ in \eqref{eq:H_1_A_a_characterization} and \eqref{eq:hardy} depends only on the circulation of the magnetic potential $A_a$ and remains finite whenever the circulation is not an integer.

As a consequence of the continuous embedding, we have the following.
\begin{lemma} \label{lemma:compactnessinversemagneticlaplacian}
Let $Im$ be the compact immersion of $\mathcal{D}^{1,2}_{A_a}(\Omega) $ into $(\mathcal{D}^{1,2}_{A_a}(\Omega) )^{\prime}$. Then, the operator $((i \nabla + A_{a})^{2})^{-1} \circ Im : \mathcal{D}^{1,2}_{A_a}(\Omega) \rightarrow \mathcal{D}^{1,2}_{A_a}(\Omega) $ is compact.
\end{lemma}
As $((i \nabla + A_{a})^{2})^{-1}$ is also self-adjoint and positive in $(\mathcal{D}^{1,2}_{A_a}(\Omega))^{\prime} $, we deduce that the spectrum of $(i \nabla + A_{a})^{2}$ in $\mathcal{D}^{1,2}_{A_a}(\Omega) $ consists of a diverging sequences of real positive eigenvalues, having finite multiplicity. 

We recall a result in \cite{HHHO1999} stating that the magnetic operator \eqref{eq:magnetic_operator_def} is equivalent to the standard Laplacian in the double covering.
\begin{lemma}[{\cite[Lemma 3.3]{HHHO1999}}]\label{lemma:gauge_invariance}
Suppose that $A_a$ has the form \eqref{eq:magnetic_potential_definition}. Then the function
\[
e^{-i\theta(y)}\varphi_j^a(y^2+a) \quad \text{defined in } \{y\in\C:\ y^2+a\in\Omega\},
\]
(here $\theta$ is the angle of the polar coordinates) is real valued and solves the following equation on its domain
\[
-\Delta(e^{-i\theta(y)}\varphi_j^a(y^2+a))=\lambda_j^a p'(y) e^{-i\theta(y)}\varphi_j^a(y^2+a), \qquad p'(y)=4|y|^2 p(y^2+a) .
\] 
\end{lemma}
As a consequence, the magnetic eigenfunctions behave, up to a complex phase, as the laplacian eigenfunctions far from the singular point $a$. The behavior near $a$  is, up to a complex phase, that of the square root of an elliptic eigenfunction. Since we are interested in the shape of the nodal lines, let us first recall the known results concerning the elliptic eigenfunctions in the plane.

\begin{proposition}[{\cite[equation (5'')]{HW}, \cite[Theorem 2.1]{HHT2009}}]\label{proposition:asymptotic_laplacian_eigenfunctions}
Let $\Omega\subset\R^2$ open, bounded, simply connected and of class $C^\infty$. Let $p$ satisfy \eqref{eq:weight_assumptions}.
If $\varphi_j$ has a zero of order $h/2$ at $0\in\overline{\Omega}$, then $h$ is even and we have
\[
\varphi_j(r,\theta)=r^{h/2} \left[ c_h\cos\left(\frac{h}{2}\theta\right) + d_h \sin\left(\frac{h}{2}\theta\right) \right] + g(r,\theta),
\]
with $x=re^{i\theta}\in\Omega$, $c_h^2+d_h^2\neq0$ and
\begin{equation}\label{eq:g_remainder}
\lim_{r\to0} \frac{\|g(r,\cdot)\|_{C^1(\partial D_r)}}{r^{h/2}}=0.
\end{equation}
In addition, there is a positive radius $R$ such that 
\begin{itemize}
\item[(i)] if $0\in\Omega$ then $(\varphi_j)^{-1}(\{0\})\cap D_R(0)$ consists of $h$ arcs of class $C^\infty$, whose tangent lines divide the disk into $h$ equal sectors;
\item[(ii)] if $0\in\partial\Omega$ then $(\varphi_j)^{-1}(\{0\})\cap D_R(0)\cap\Omega$ consists of $h/2-1$ arcs of class $C^\infty$, whose tangent lines divide the half disk into $h/2$ equal sectors.
\end{itemize}
\end{proposition}
The behavior near the boundary at point (ii) above can be deduced from point (i). Indeed, since the boundary is regular, it can be locally rectified as described in Lemma \ref{lemma:Riemannmapping} below. By performing an odd extension, we transform the boundary point into an interior point of type (i).

We summarize below the local properties of the magnetic eigenfunction near the pole. The proofs can be found in \cite[Theorem 1.3]{FFT2011}, \cite[Theorem 2.1]{HHHO1999} and \cite[Theorem 1.5]{TN1}.

\begin{proposition}\label{proposition:asymptotic_expansion_eigenfunction}
There exists an odd integer $h\geq1$ such that $\varphi_j^a$ has a zero of order $h/2$ at $a$. Moreover, the following asymptotic expansion holds near $a$
\begin{equation}\label{eq:varphia_asymptotic_expansion}
\varphi_j^a(|x-a|,\theta_a)=e^{i\frac{\theta_a}{2}}|x-a|^{h/2} \left[ c_h(a)\cos\left(h\frac{\theta_a}{2}\right) + d_h(a) \sin\left(h\frac{\theta_a}{2}\right) \right] + g(|x-a|,\theta_a)
\end{equation}
where $x-a=|x-a|e^{i\theta_a}$, $c_h(a)^2+d_h(a)^2\neq0$ and $g$ satisfies \eqref{eq:g_remainder}.
In addition, there is a positive radius $R$ such that $(\varphi_j^a)^{-1}(\{0\})\cap D_R(a)$ consists of $h$ arcs of class $C^\infty$. If $h\geq 3$ then the tangent lines to the arcs at the point $a$ divide the disk into $h$ equal sectors.
\end{proposition}

\section{Equation on a domain with locally rectified boundary}\label{sec:rectified_boundary}
The local analysis near $0\in\partial\Omega$ is easier if the boundary is locally flat at $0$, i.e. if there exists $r>0$ such that
 \begin{align}\label{eq:rectified_domain}
\Omega \cap D_r(0) = \left\{ x \in D_r(0) \, : \, x_1 > 0 \right\} =: D_r^+(0) .
\end{align}
If $\partial\Omega$ is sufficiently regular, we can locally rectify it without altering the nature of the problem, as we show in the following lemma. This is not the case when $\partial\Omega$ presents an angle, as for the angular sector presented in the Introduction, see Appendix \ref{appendix:cone}.

\begin{lemma} \label{lemma:Riemannmapping}
Let $\Omega,\Omega^{\prime} \subset \mathbb{C}$ be open, bounded and simply connected domains with $C^{\infty}$ boundary. There exist a conformal transformation $\Phi \, : \, \Omega^{\prime} \rightarrow \Omega$, $\Phi \in C^{\infty}(\overline{\Omega^{\prime}})$ and a function $\chi \in C^{\infty}(\overline{\Omega^{\prime}})$ such that, letting $w_k^{a'} = e^{-i \chi} (\varphi_k^a \circ \Phi)$,
we have
\begin{align*}
(i \nabla + A_{a'})^2 w_k^{a'} = \lambda_k^a p'(x) w_k^{a'}, \quad 
w_k^{a'} \in \mathcal{D}^{1,2}_{A_{a'}}(\Omega^{\prime}),
\end{align*}
where $\Phi(a') = a$ and $p'$ satisfies \eqref{eq:weight_assumptions}.
\end{lemma}
\begin{proof}
The existence of a regular conformal map is ensured by the Riemann Mapping Theorem. We note that $\Phi$ is regular up to the boundary thanks to the assumption that $\Omega,\Omega^{\prime}$ have $C^\infty$ boundary.

First define $v_k^{a'} = \varphi_k^a \circ \Phi$. This function solves
\begin{align*}
(i \nabla + B_{a'})^2 v_k^{a'} = \lambda_k^a p'(x) v_k^{a'},
\end{align*}
where the weight is given by $p'(x) = |\Phi^{\prime}|^2 p \circ \Phi(x)$, $B_{a'} = ( \Phi^{\prime \, t} \cdot A_a) \circ \Phi$ and where $\Phi^{\prime}$ is the matrix of the derivatives of $\Phi$. Indeed, remember that the magnetic potential $A_a$ is defined as $A_a = \nabla \theta_a /2$ almost everywhere in $\Omega$. Under the conformal transformation, the magnetic potential will transform as
\begin{align*}
B_{a'}  = \frac{\nabla ( \theta_a \circ \Phi)}{2}=(\Phi^{\prime \, t} \cdot A_a)\circ \Phi.
\end{align*}

Next, we can verify explicitly that this new magnetic potential $B_{a'}$ has the same circulation than $A_a$. Indeed, if we consider a closed path $\gamma$, winding once around the point $a'$ in $\Omega^{\prime}$, we obtain
\begin{align*}
& \frac{1}{2 \pi} \oint_{x \equiv \gamma} B_{a'}^t(x)|_{x=\gamma(t)} \cdot \gamma^{\prime}(t) \, \mathrm{d}t = \frac{1}{2 \pi} \oint_{x \equiv \gamma} (A_a^t \cdot \Phi^{\prime}) (\Phi(x))|_{x = \gamma(t)} \cdot \gamma^{\prime} (t) \, \mathrm{d}t \\
& = \frac{1}{2 \pi} \oint_{y \equiv \sigma} (A_a^t \cdot \Phi^{\prime} )(y) |_{y = \sigma(t)} \cdot \Phi^{\prime \, -1} \sigma^{\prime}(t) \, \mathrm{d}t = \frac{1}{2\pi} \oint_{y \equiv \sigma} A_a^t(y)|_{y=\sigma(t)} \cdot \sigma^{\prime}(t) \, \mathrm{d}t = \frac{1}{2},
\end{align*}
where we used the change of variable $y = \Phi(x)$. The new path $\sigma = \Phi \circ \gamma$ is a closed path winding once around $a$ in $\Omega$ and its derivative is $\gamma^{\prime}= \Phi^{\prime \, -1} \cdot \sigma^{\prime}$. Moreover, we can verify that $\nabla \times B_{a'} =0$ everywhere in $\Omega^{\prime} \backslash \{ a' \}$. Then, there exists $\chi \in C^{\infty}(\overline{\Omega^{\prime}})$ such that $B_{a'} = A_{a'} + \nabla \chi$, where $A_{a'}$ has the form given in \eqref{eq:magnetic_potential_definition}.

Finally, we can gauge away the term $\nabla \chi$ by letting $w_k^{a'} = e^{-i\chi} v_k^{a'}$.
\end{proof}

As a consequence, by simply renaming the weight $p(x)$ in equation \eqref{eq:main_with_weight}, we will work in a new domain satisfying \eqref{eq:rectified_domain}. 

\begin{remark}\label{rem:regularity_Omega}
In fact, in what follows, it will be enough to have a weight $p$ of class $C^1$. Thanks to \cite[Theorem 5.2.4]{KrantzBook2006}, to this aim it is sufficient to assume $\partial\Omega\in C^{2,\gamma}$ for some positive $\gamma$. This ensures that $\Phi$ in the previous lemma is $C^2(\overline{\Omega'})$, so that the transformed weight $p'(x) = |\Phi^{\prime}|^2 p \circ \Phi(x)$ is $C^1(\overline{\Omega'})$. 
\end{remark}

In Sections \ref{sec:freq_formula} and \ref{sec:pole_approaching_not_nodal}, we will study the behavior of the eigenfunctions $\varphi_j^a$ as $a = (a_1, a_2)$ approaches $0 \in \partial \Omega$. As we will see later, the significant parameter will be the distance of $a$ from the boundary $\partial \Omega$. Such distance is $a_1$ if $\partial \Omega$ is locally flat at $0$ and $|a|$ is sufficiently small. We will perform the analysis in half balls centered at $(0,a_2)$; all the future estimates will be independent from $a_2$. To this aim, it will be useful to have some general inequalities for functions $u \in H^1_{A_a}(D_r^+(0,a_2))$ with $u=0$ on $\{ x_1=0 \}$, $a \in D_r^+(0,a_2)$. 
To simplify the notation, we write 
\[
\pi(a)=(0,a_2),
\]
$\pi$ corresponding then to the projection onto the $x_2$-axis, so that
\[
D_r^+(\pi(a))=\{ (x_1,x_2)\in\R^2:\, x_1^2+(x_2-a_2)^2<r^2, \, x_1>0 \}.
\]

\begin{lemma}[Poincaré inequality] \label{lemma:Poincaré}
Let $a \in D_r^+(\pi(a))$. For all $u \in H^{1}_{A_a}(D_{r}^{+}(\pi(a)))$, with $u=0$ on $\{x_1=0\}$, the following inequality is verified
\begin{align} \label{eq:Poincaré}
\frac{1}{r^{2}} \int_{D_{r}^{+}(\pi(a))} |u|^{2} \, \mathrm{d}x \leq \frac{1}{r} \int_{\partial D_{r}^{+}(\pi(a))} |u|^{2} \, \mathrm{d} \sigma + \int_{D_{r}^{+}(\pi(a))} |(i \nabla + A_a) u|^{2} \, \mathrm{d}x.
\end{align}
\end{lemma}
\begin{proof}
By explicit calculation we see that, for every $u \in H^1_{A_a}(D_r^+(\pi(a)))$, we have
\begin{equation}\label{eq:divergence_u_square_x}
\nabla \cdot (|u|^{2} (x - \pi(a))) = 2 Re \left( i u \, \overline{ (i \nabla + A_b) u \cdot (x - \pi(a)) } \right) + 2 |u|^{2} \qquad \text{a.e. }x\in D_r^+(\pi(a)).
\end{equation}
Then
\begin{align*}
\int_{D_{r}^{+}(\pi(a))} |u|^{2} \dx & = - Re \left( i \int_{D_{r}^{+}} u \, \overline{ (i\nabla + A_a) u \cdot (x - \pi(a)) } \dx \right) + \frac{r}{2} \int_{\partial D_{r}^{+}(\pi(a))} |u|^{2} \dsigma \\
                                      & \leq \frac{1}{2} \int_{D_{r}^{+}(\pi(a))} |u|^{2}\dx + \frac{r^2}{2} \int_{D_r^+(\pi(a))}|(i \nabla + A_a) u|^{2} \dx + \frac{r}{2} \int_{\partial D_{r}^{+}(\pi(a))} |u|^{2} \dsigma,
\end{align*}
where we used the Young inequalities and the fact that $(x - \pi(a)) = r \nu$ on $\partial D_r^+(\pi(a))$. This proves the statement.
\end{proof}

Similarly, we can prove the similar Poincaré inequality for all functions $v$ in $H^1(D_r^+(0),\R)$ with zero boundary conditions on $\{x_1 = 0\}$,
\begin{align} \label{eq:Poincaré2}
\frac{1}{r^{2}} \int_{D_{r}^{+}(0)} |v|^{2} \, \mathrm{d}x \leq \frac{1}{r} \int_{\partial D_{r}^{+}(0)} |v|^{2} \, \mathrm{d} \sigma + \int_{D_{r}^{+}(0)} | \nabla v|^{2} \, \mathrm{d}x.
\end{align}

\begin{lemma} \label{lemma:inequality2}
Let $a \in D_r^+(\pi(a))$. For $u \in H^{1}_{A_a}(D_{r}^{+}(\pi(a)))$, with $u=0$ on $\{x_1=0\}$,
the following holds
\begin{align} \label{eq:inequality2}
\frac{1}{r} \int_{\partial D_{r}^{+}(\pi(a))} |u|^{2}\dsigma \leq \int_{D_{r}^{+}(\pi(a))} |(i \nabla + A_a) u|^{2}\dx.
\end{align}
\end{lemma}
\begin{proof}
We will prove the following statement: for all $v \in H^1(D_r^+(\pi(a)),\R)$, with $v = 0$ on $\{x_1 = 0\}$, we have
\begin{align} \label{eq:Poincaretype}
\frac{1}{r} \int_{\partial D_{r}^{+}(\pi(a))} |v|^{2}\dsigma \leq \int_{D_{r}^{+}(\pi(a))} |\nabla v|^{2}\dx.
\end{align}
The lemma follows from it by taking $v=|u|$ and by applying the diamagnetic inequality.
It is sufficient to prove \eqref{eq:Poincaretype} in the ball $D_1^+(0)$ since the general case can be recovered by performing a translation and a dilation.
Let
\begin{align*}
\beta = \inf \left\{ \int_{D_{1}^{+}(0)} | \nabla w |^{2}\dx \, : \, \int_{\partial D_{1}^{+}(0)} w^{2}\dsigma = 1 \, , \, w \in H^1(D_1^+(0),\mathbb{R}) \, , \, w= 0 \text{ on } \{x_1=0\} \right\}.
\end{align*}
Let $w_{k}$ be a minimizing sequence. Then $\sup_{k} \int_{D_{1}^{+}(0)} |\nabla w_{k}|^{2}\dx \leq C$ and by the Poincaré inequality \eqref{eq:Poincaré2} we have $\sup_{k} \| w_{k} \|_{H^{1}(D_{1}^{+}(0),\mathbb{R})} \leq C'$. Hence there exists a function $\bar{w} \in H^{1}(D_{1}^{+}(0),\mathbb{R})$ such that, up to a subsequence,
\begin{align*}
w_{k} \rightharpoonup \bar{w} \ \mbox{ in } \ H^{1}(D_{1}^{+}(0),\mathbb{R})
\quad \text{ and } \quad w_{k} \rightarrow \bar{w} \ \mbox{ in } \ L^{2}(D_{1}^{+}(0),\mathbb{R})
\end{align*}
by the compact injection. Using the lower semi-continuity of the $H^1$-norm we obtain
\begin{align*}
\int_{D_{1}^{+}(0)} |\nabla \bar{w} |^{2}\dx \leq  \liminf_{k \to + \infty} \int_{D_{1}^{+}(0)} |\nabla w_{k}|^{2}\dx = \beta.
\end{align*}
By the compact embedding $H^1(\Omega)\hookrightarrow L^2(\partial \Omega)$ (see e.g. \cite{AdamsFournierBook2003}) we also have $\int_{\partial D_{1}^{+}(0)} |\bar{w}|^{2}\dsigma = 1$.
Therefore $\bar{w}$ is a minimizer and solves the associated Euler-Lagrange equation
\begin{align*}
\left\{ \begin{aligned}  - \Delta \bar{w} & = 0 \quad & D_{1}^{+}(0) \\
                         \frac{\partial \bar{w}}{\partial \nu} & = \beta \bar{w} \quad & \partial D_{1}^{+}(0) \\
                         \bar{w}&=0 \quad & x_1=0.
                        \end{aligned} \right.
\end{align*}
If we decompose the boundary trace in Fourier series $\bar{w}(1,\theta)=\sum_{k\text{ odd}} \alpha_k \cos(k\theta)$, $\theta\in [-\frac{\pi}{2},\frac{\pi}{2}]$, then $\bar{w}(r,\theta)=\sum_{k\text{ odd}} \alpha_k r^k \cos(k\theta)$. The boundary conditions imply $k\alpha_k=\beta \alpha_k$ for every $k\geq1$. We deduce that $\bar{w}(r,\theta)=\alpha_\beta r^\beta e^{i\beta\theta}$ for some integer $\beta\geq 0$. Since $\bar{w}$ can not vanish because of the constraint, the infimum is assumed by $\beta=1$ and $\bar{w}=\sqrt{\frac{2}{\pi}}x_1$.
\end{proof}

\section{Pole approaching the boundary on a nodal line of $\varphi_k$}\label{sec:pole_approaching_nodal}

In this section we prove Theorem \ref{theorem:at_least_one_nodal_line}. We use some of the ideas introduced in \cite[Section 6]{BonnaillieNorisNysTerracini2013}. As we already mentioned, we modify the argument therein both in order to avoid the use of local inversion methods and in order to prove that the convergence $\lambda_k^a\to\lambda_k$ takes place from below.

We adopt the standard notation $f(x)=o( g(x))$ as $x\to x_0$ if $\lim_{x\to x_0} f(x)/g(x)$ is zero, $f(x)=O(g(x))$ as $x\to x_0$ if $\limsup_{x\to x_0} |f(x)/g(x)|$ is finite, $f(x)\sim g(x)$ as $x\to x_0$ if $\lim_{x\to x_0} f(x)/g(x)$ is finite and different from zero.

\begin{lemma}\label{lemma:v_i_int_estimates}
Let $\lambda>0$ and $p$ satisfy \eqref{eq:weight_assumptions}. Let $a \to 0$ so that $\frac{a}{|a|} \to e \notin \{ x_1 = 0 \}$. Consider the following set of equations in the parameter $a$
\begin{equation}\label{eq:v}
\left\{\begin{array}{ll}
(i\nabla+A_a)^2v=\lambda p(x) v \quad & \text{in } D_{2|a|}^+(0) \\
v=0 & \text{on } \{x_1=0\} \\
v=e^{i\frac{\theta_a}{2}}  \left\{ |a|^{1+n} f+g(2|a|,\cdot) \right\} \quad &\text{on } \partial D_{2|a|}^+(0)\cap \R^2_+,
\end{array}\right.
\end{equation}
where $f,g(2|a|,\cdot)\in H^1(\partial D_{2|a|}^+(0)\cap\R^2_+)$ are real valued, vanish at $-\pi/2$ and at $\pi/2$, $f\not\equiv0$ and, for some $n\in \N$,
\[
\lim_{|a|\to0} \frac{\|g(2|a|,\cdot)\|_{H^1(\partial D_{2|a|}^+(0)\cap\R^2_+)}}{|a|^{1+n}}=0.
\]
Then for $|a|$ sufficiently small there exists a unique solution of \eqref{eq:v}, which moreover satisfies
\begin{gather*}
\|v\|_{L^2(\partial D^+_{2|a|}(0))}\sim |a|^{\frac{3}{2}+n },\
\|v\|_{L^2(D^+_{2|a|}(0))}\sim |a|^{2+n }, \\
\|(i\nabla+A_a)v\cdot\nu\|_{L^2(\partial D^+_{2|a|}(0)\cap\R^2_+)} \sim |a|^{\frac{1}{2}+n },\ 
\left| \int_{\partial D_{2|a|}^+(0)} (i\nabla+A_a)v\cdot\nu v\dsigma \right| \sim |a|^{2n+2}.
\end{gather*}
\end{lemma}

\begin{proof}
Notice that the boundary trace is continuous on $\partial D^+_{2|a|}(0)$. Indeed, we can choose $\theta_a$ to be discontinuous on the segment joining $a$ with the origin, so that $e^{i\theta_a/2}$ restricted to the boundary is discontinuous only at the origin, where the boudary trace vanishes.
The existence and uniqueness of the minimizer follow plainly by the fact that the quadratic form
\begin{equation}\label{eq:quadratic_form}
\int_{D_{2|a|}^+(0)} \left[ |(i\nabla+A_a)v|^2-\lambda p(x) |v|^2 \right]\dx
\end{equation}
is coercive for $|a|$ sufficiently small. 

The estimate on the $L^2(\partial D^+_{2|a|}(0))$-norm is immediate. In order to prove the remaining estimates, we wish to proceed as in \cite[Lemma 6.1]{BonnaillieNorisNysTerracini2013}. To this aim, we first make the rescaling $x/|a|$ as well as a normalization of the functions by the factor $|a|^{1+n}$. Next, in order to have a fixed singularity, not depending on the parameter $a$, we apply a conformal transformation. Once the singularity is fixed, by going on the double covering, we obtain a family of elliptic equations defined on the same domain. More precisely, we consider the family of conformal transformations $\Phi_a$ in the parameter $a$ satisfying the properties
\begin{align*}
\Phi_a^{-1} \in C(\overline{D_2^+(0)}), \quad \Phi_a^{-1}(D_2^+(0))=D_2^+(0), \quad 
 \text{ and} \quad \Phi_a^{-1}\left(\frac{a}{|a|}\right) = e. 
\end{align*}
Since $a/|a| \to e$, we have that $\Phi_a$ is a small perturbation of the identity. Moreover, thanks to the fact that $e \notin \{ x_1 = 0\}$, we obtain that 
\begin{align} \label{eq:unifrom_bound_Phi}
| \Phi_a^{\prime}| \leq C \text{ uniformly in }  a.
\end{align}
We decompose $v$ in the following way
\[
v(x)=|a|^{1+n} z_1\left(\Phi_a^{-1}\left(\frac{x}{|a|}\right)\right) + |a|^{1+n} z_2\left(\Phi_a^{-1}\left(\frac{x}{|a|}\right)\right),
\]
where 
\[
(i\nabla+A_e)^2 z_1=0 \text{ in } D^+_2(0), \quad z_1=|a|^{-(1+n)}v(|a|\Phi_a) \text{ on } \partial D_2^+(0),
\]
\[
(i\nabla+A_e)^2 z_2 = \lambda \, p(|a|\Phi_a)\, |\Phi_a'|^2 \,|a|^2 (z_1 + z_2) \text{ in } D^+_2(0) , \quad
z_2=0  \text{ on } \partial D_2^+(0).
\]
Now we have two equations with variable coefficients on a fixed domain and with fixed singularity, so that the double covering
\begin{align*}
\tilde{\Omega} = \left\{ y \in \mathbb{C} \, : \, y^2 + e \in D_2^+(0) \right\}
\end{align*}
does not depend on $a$.
Since $e^{i \theta_a(|a|\Phi_a)/2} = e^{i \chi} e^{i \theta_e/2}$, for some regular function $\chi$, Lemma \ref{lemma:gauge_invariance} implies that the new functions 
\begin{align*}
\tilde{z}_i(y)  = e^{-i\frac{\theta_e}{2} - i \chi} z_i(y^2 +e)
\end{align*}
solve an elliptic equation in $\tilde{\Omega}$. Proceeding as in the proof of \cite[Lemma 6.1]{BonnaillieNorisNysTerracini2013} and using \eqref{eq:unifrom_bound_Phi}, one can show that $\tilde{z}_2$ provides a negligible contribution in the sense that
\[
\|\tilde{z}_2\|_{H^1(\tilde{\Omega})}+\|\nabla \tilde{z}_2\cdot\nu\|_{L^2(\partial \tilde{\Omega})} \leq C|a|^2,
\]
and moreover that
\[
c_1 \leq \|\tilde{z}_1\|_{H^1(\tilde{\Omega})} \leq c_2 , \qquad 
c_1 \leq \| \nabla \tilde{z}_1 \cdot \nu \|_{L^2(\partial \tilde{\Omega})} \leq c_2,
\]
for positive constants $c_1, c_2$. This also provides
\[
c_1 \leq \left|\int_{\partial \tilde{\Omega}} \nabla (\tilde{z}_1+\tilde{z}_2)\cdot\nu(\tilde{z}_1+\tilde{z}_2) \dsigma \right| \leq c_2.
\]
Going back to the original domain $D_{2|a|}^+(0)$, we obtain the statement.
\end{proof}

\begin{lemma} \label{lemma:greatesteigenvalue2}
Let $k\geq 2$ and let $M = M(|a|) = (m_{ij})$ be $k \times k$ real valued matrices depending in a smooth way on the parameter $|a|$. Suppose that there exist $n\in\N_0$ and  $C_k > 0$ such that, as $|a| \to 0$, we have
\begin{align*}
& m_{ii} = \lambda_i + O(|a|^2), \, i = 1, \ldots, k-1,  \quad m_{kk} = \lambda_k - C_{k}|a|^{2n + 2} + o(|a|^{2n+2}), \\
& m_{ij} = O(|a|^2), \, i \neq j, \, i,j = 1, \ldots, k-1, \quad m_{ij} = O(|a|^{n+2}), \, i \neq j \text{ and } i =k \text{ or } j = k. 
\end{align*}
If $\lambda_{k-1} < \lambda_k$, then the greatest eigenvalue of $M$ satisfies
\begin{align*}
\lambda_{max}(M) = \lambda_k - C_k|a|^{2n+2} + o(|a|^{2n+2}) \text{ as } a \to 0.
\end{align*}
\end{lemma}

\begin{proof}
In order to estimate the eigenvalues of $M$, we compute the determinant of the matrix $B=M-tId=(b_{ij})$. We have
\begin{equation}\label{eq:determinant}
\text{det}(B)=\sum_{\sigma\in P_k} \text{sgn}(\sigma) \prod_{i=1}^k
b_{i\sigma(i)},
\end{equation}
where $\sigma$ is a permutation of the set $\{1,\ldots,k\}$, $P_k$ is the set of all such permutations, $\text{sgn}(\sigma)$ is the sign of $\sigma$ and $\sigma(i)$ is the image of the element $i\in \{1,\ldots,k\}$ under the action of $\sigma$. We recall that a fixed point of $\sigma$ is an element $i$ such that $\sigma(i)=i$. We define, for $r=0,\ldots,k$,
\[
P_{k,r}=\{\sigma\in P_k: \ \sigma \text{ has exactly } r \text{ fixed points} \}.
\]
Notice that $P_{k,k}=\{Id\}$ and $P_{k,k-1}=\emptyset$.
We split the sum in \eqref{eq:determinant} in the following way:
\[
\text{det}(B)=\prod_{i=1}^k b_{ii} + 
\sum_{r=1}^{k-2}\sum_{\substack{\sigma\in P_{k,r} \\ \sigma(k)=k}} \text{sgn}(\sigma) \prod_{i=1}^k b_{i\sigma(i)} 
+ \sum_{r=0}^{k-2}\sum_{\substack{\sigma\in P_{k,r} \\ \sigma(k)\neq k}} \text{sgn}(\sigma) \prod_{i=1}^k b_{i\sigma(i)}.
\]
Due to the specific form of $M$ we can estimate each piece as follows
\[
\prod_{i=1}^k b_{ii}=(\lambda_k-C_k|a|^{2n +2}+o(|a|^{2n+2})-t) \prod_{i=1}^{k-1} (\lambda_i+O(|a|^{2})-t),
\]
\[
\sum_{\substack{\sigma\in P_{k,r} \\ \sigma(k)=k}} \text{sgn}(\sigma) \prod_{i=1}^k b_{i\sigma(i)} = (\lambda_k-C_k|a|^{2n +2}-t) O\left(|a|^{2(k-r)}\right) Q_{r-1}(t,|a|^2), \quad r=1,\ldots,k-2,
\]
\[
\sum_{\substack{\sigma\in P_{k,r} \\ \sigma(k)\neq k}} \text{sgn}(\sigma) \prod_{i=1}^k b_{i\sigma(i)} = O\left(|a|^{2(n+2)}\right) O\left(|a|^{2(k-r-2)}\right) Q_r(t,|a|^2), \quad r=0,\ldots,k-2,
\]
where $Q_{r}(t,|a|^2)$ denotes a polynomial of degree $r$ in the variable $t$, which depends on $|a|$ with terms of order $O(|a|^2)$. More explicitly, $Q_{r}(t,|a|^2)$ is given by the sum over any possible choice of $r$ numbers in the set $\{1,\ldots,k-1\}$ of a product of $r$ terms of the form $(\lambda_i + O(|a|^2) - t)$, for some $i \neq k$. We can also define $Q_{k-1}(t,|a|^2) = \prod_{i=1}^{k-1} (\lambda_i+O(|a|^{2})-t)$ and remark that 
\begin{align}  \label{eq:Q_k-1_not_vanish}
Q_{k-1}(\lambda_k, 0) = \prod_{i=1}^{k-1} (\lambda_i - \lambda_k) \neq 0,
\end{align}
where we used the assumption that $\lambda_{k-1} < \lambda_k$.

Let $\varepsilon=|a|^{2}$. We rewrite the determinant in terms of $\varepsilon$, obtaining
\[
\begin{split}
\text{det}(B)=(\lambda_k-C_k\varepsilon^{n +1}+o(\varepsilon^{n+1})-t) 
\left\{ Q_{k-1}(t,\varepsilon)+\sum_{r=1}^{k-2} O\left(\varepsilon^{k-r}\right) Q_{r-1}(t,\varepsilon) \right\} \\
+ \sum_{r=0}^{k-2} O\left(\varepsilon^{n+k - r}\right) Q_r(t,\varepsilon) =: f(t,\varepsilon).
\end{split}
\]
The assumptions of the implicit function theorem hold for $f(t,\varepsilon)$ at the point $(\lambda_k,0)$. Indeed, $f(\lambda_k, 0) = 0$, $f$ is at least $C^{n+1}$ in a neighbourhood of $(\lambda_k, 0)$, and $\frac{\partial f}{\partial t}(\lambda_k,0) = - Q_{k-1}(\lambda_k,0) \neq 0$ thanks to  \eqref{eq:Q_k-1_not_vanish}. Then there exists a function $\lambda(\varepsilon) \in C^{n+1}$, defined in a neighbourhood of $\varepsilon = 0$, such that $f(\lambda(\varepsilon), \varepsilon)=0$. 

Let us first differentiate this relation with respect to $\varepsilon$ and estimate it in $(\lambda_k, 0)$
\begin{align*}
\frac{\partial f}{\partial t}(\lambda_k,0) \, \lambda'(0) + \frac{\partial f}{\partial \varepsilon} (\lambda_k,0) = 0.
\end{align*}
Since $n \geq 1$, $\frac{\partial f}{\partial \varepsilon}(\lambda_k,0) = 0$ and we conclude that $\lambda'(0) = 0$. We can differentiate $n+1$ times the identity $f(\lambda(\varepsilon), \varepsilon)=0$ and each time use the relations obtained in the previous step. We have
\begin{align*}
\frac{\partial f}{\partial t}(\lambda_k,0) \, \lambda^{(j)}(0) + \frac{\partial^{j}f}{\partial \varepsilon^j}(\lambda_k,0) =0, \ j=1,\ldots,n+1.
\end{align*}
Thanks to the fact that $n + k - r > n + 1$ for all $r=0,\ldots,k-2$, and using \eqref{eq:Q_k-1_not_vanish}, we deduce
\begin{align*}
\frac{\partial^{j}f}{\partial \varepsilon^j}(\lambda_k,0) =0, \ j=1,\ldots,n, \qquad
\frac{\partial^{n+1}f}{\partial \varepsilon^{n+1}} (\lambda_k,0) = - C_k (n+1)! Q_{k-1}(\lambda_k,0) \neq 0.
\end{align*}
Then
\begin{align*}
\lambda^{(j)}(0) = 0 , \ j=1,\ldots,n, \qquad \lambda^{(n+1)}(0) = - C_k (n+1)! 
\end{align*}
and we conclude that $\lambda(\varepsilon) = \lambda_k - C_k \varepsilon^{n+1} + o(\varepsilon^{n+1})$ as $\varepsilon \to 0$.
\end{proof}

\begin{proof}[Proof of Theorem \ref{theorem:at_least_one_nodal_line}]
We can assume without loss of generality that $b=0$ and moreover, by Lemma \ref{lemma:Riemannmapping}, that $\Omega$ satisfies \eqref{eq:rectified_domain}. Let $\varphi_k$ have a zero of order $h /2\geq2$ at $0$, corresponding to
\begin{equation}\label{eq:n_h_relation}
n =\frac{h }{2} -1 \geq 1
\end{equation}
arcs of nodal line ending at $0$. Denote by $\Gamma$ any such piece of nodal line and let $a\in\Gamma$. We shall take advantage of the min-max characterization of eigenvalues, which we exploit by constructions suitable finite dimensional spaces of competitor functions. 

{\bf Step 1.} Construction of the space of competitors.
As shown in \cite[Lemma 4.1]{BonnaillieNorisNysTerracini2013}, we can choose the discontinuity of $\theta_a$ on the piece of $\Gamma$ connecting $a$ with the origin, so that
\[
\frac{\nabla \theta_a}{2}=A_a \quad \text{ globally in } \Omega\setminus D_{|a|}^+(0).
\]
For $i=1, \ldots, k$ we define
\begin{align}\label{eq:v_ext_def}
v_{i}^{ext}(x) = e^{i \frac{\theta_{a}}{2}(x)} \varphi_{i}(x), \qquad x\in \Omega\setminus D_{2|a|}^+(0).
\end{align}
Since $e^{i \frac{\theta_{a}}{2}}$  is univalued and regular in $\Omega \backslash D^+_{|a|}(0)$, the gauge invariance implies
\begin{equation}\label{eq:equationuext2}
 (i \nabla + A_{a})^{2} v_{i}^{ext} = \lambda_{i} p(x) v_{i}^{ext} \quad \text{in } \Omega \backslash D^+_{2|a|}(0).
 \end{equation}
In the interior of the small disk we take the solution of the magnetic equation having the same boundary trace, that is, for $i=1,\ldots,k$,
\begin{align}\label{eq:v_k_int_equation}
(i \nabla + A_{a})^{2} v_{i}^{int} = \lambda_{i} p(x) v_{i}^{int} \quad \text{in }  D^+_{2|a|}(0), \qquad 
v_{i}^{int} = e^{i \frac{\theta_{a}}{2}} \varphi_{i} \quad \text{on } \partial D^+_{2|a|}(0).
\end{align}
By uniqueness, $v_i^{int}$ can also be characterized as the function which achieves
\begin{align}\label{eq:v_int_def}
\inf\left\{ \int_{D_{2|a|}^+(0)} \left[ |(i\nabla+A_a)v|^2-\lambda_i p(x) |v|^2 \right]\dx : \ v\in H^1_{A_a}(D^+_{2|a|}(0)), \ v=e^{i \frac{\theta_{a}}{2}} \varphi_{i} \text{ on } \partial D^+_{2|a|}(0) \right\}.
\end{align}

Though $v_i^{int}$ and $v_i^{ext}$ solve the same equation on the respective domains, the competitor functions defined as
\begin{align}\label{eq:v_i_def}
v_{i} = \left\{ \begin{aligned} & v_{i}^{int} \quad & D^+_{2|a|}(0) \\
                                & v_{i}^{ext} \quad & \Omega \backslash D^+_{2|a|}(0)  \end{aligned} \right.
\end{align}
do not solve the equation in $\Omega$. Indeed, we have, for all $\phi \in \mathcal{D}^{1,2}_{A_a}(\Omega)$,
\begin{align}\label{eq:v_i_equation}
\int_{\Omega} \left[ (i \nabla + A_{a}) v_{i} \cdot \overline{(i \nabla + A_{a}) \phi} - \lambda_{i} p(x) v_{i} \overline{\phi}\right]\dx 
= i \int_{\partial D^+_{2|a|}(0)} (i \nabla + A_{a}) (v_{i}^{ext} - v_{i}^{int} ) \cdot \nu \overline{\phi} \dsigma,
\end{align}
where we used the formula of integration by parts, \eqref{eq:equationuext2} and \eqref{eq:v_k_int_equation}.

{\bf Step 2.} Estimates on the single competitor functions. By Proposition \ref{proposition:asymptotic_laplacian_eigenfunctions}, $\varphi_i$ has the following behavior on $\partial D^+_{2|a|}(0)$, for $i<k$,
\[
\varphi_i|_{\partial D^+_{2|a|}(0)}= |a| c_1 \cos\theta + o(|a|),\quad \text{as } a\to0,
\]
with $c_1$ eventually $0$, whereas for $\varphi_k$ we have
\[
\varphi_k|_{\partial D^+_{2|a|}(0)}= |a|^{1+n } \left(c_{1+n } \cos[(1+n )\theta]+d_{1+n }\sin[(1+n )\theta]\right) + o(|a|^{1+n }), \quad \text{as } a\to0,
\]
with $c_{1+n}\neq0, d_{1+n}=0$ if $n$ is even and $c_{1+n}=0, d_{1+n}\neq0$ if $n$ is odd.
Since $a$ belongs to one of the nodal lines of $\varphi_k$, $\Gamma$, and the tangents to the nodal lines divide $\pi$ into equal angles, we have that $a/|a|\to e \not\in \{x_1=0\}$ and we recover the property \eqref{eq:unifrom_bound_Phi}. Hence Lemma \ref{lemma:v_i_int_estimates} applies providing the following estimates
\begin{equation}\label{eq:v_i_int_estimates}
\begin{split}
\|v_i^{int}\|_{L^2(\partial D^+_{2|a|}(0))}= O(|a|^{\frac{3}{2}}),\
\|v_i^{int}\|_{L^2(D^+_{2|a|}(0))}= O(|a|^2), \\
\|(i\nabla+A_a)v_i^{int}\cdot\nu\|_{L^2(\partial D^+_{2|a|}(0)\cap\Omega)}= O(|a|^{\frac{1}{2}}),
\end{split}
\end{equation}
for $i=1,\ldots,k-1$, and
\begin{gather}
\|v_k^{int}\|_{L^2(\partial D^+_{2|a|}(0))}\sim |a|^{\frac{3}{2}+n },\
\|v_k^{int}\|_{L^2(D^+_{2|a|}(0))}\sim |a|^{2+n }, \label{eq:v_k_estimate1} \\
\|(i\nabla+A_a)v_k^{int}\cdot\nu\|_{L^2(\partial D^+_{2|a|}(0)\cap\Omega)} \sim |a|^{\frac{1}{2}+n }, \ 
\int_{\partial D_{2|a|}^+(0)} (i\nabla+A_a)v_k^{int} \cdot\nu v_k^{int} \dsigma \sim |a|^{2n+2}. \label{eq:v_k_estimate2}
\end{gather}

{\bf Step 3.} We claim that there exists a constant $C_{k} > 0$ such that
\begin{align}\label{eq:boundary_integral_v_k}
i \int_{\partial D^+_{2|a|}(0)} (i \nabla + A_{a}) (v_{k}^{ext} - v_{k}^{int} ) \cdot \nu \overline{v}_k \dsigma 
= - C_{k} |a|^{2 n +2} +o(|a|^{2 n +2}).
\end{align}
The asymptotic behavior is consequence of \eqref{eq:v_k_estimate2}.
Let us prove that the quantity we want to estimate is negative. To this aim, we extend the function $v_{k}^{ext}=e^{i\frac{\theta_a}{2}}\varphi_k$ to all $\Omega$. Such extension is continuous in $D^+_{2|a|}(0)$, since $\varphi_k$ vanishes on $\Gamma$ and $\theta_a$ is regular outside $\Gamma$, and solves
\begin{align*}
\left\{ \begin{aligned} (i \nabla + A_{a})^{2} v_{k}^{ext} & = \lambda_{k} p(x) v_{k}^{ext} \quad & \Omega \backslash \Gamma \\
               v_{k}^{ext} & = 0 \quad & \partial \Omega . \end{aligned} \right.
\end{align*} 
Since $v_k^{ext}=0$ on $\Gamma$, we can test this equation by $v_k^{ext}$ itself in $D^+_{2|a|}(0)$ and apply the formula of integration by parts to obtain
\[
i \int_{\partial D^+_{2|a|}(0)} (i \nabla + A_{a}) v_{k}^{ext} \cdot \nu \overline{v_{k}} \dsigma=-\int_{D^+_{2|a|}(0)} \left[ |(i\nabla+A_a)v_k^{ext}|^2-\lambda_k p(x) |v_k^{ext}|^2 \right]\dx.
\] 
On the other hand, by testing \eqref{eq:v_k_int_equation} by $v_k^{int}$ we obtain the same expression for $v_k^{int}$. By subtracting the two equalities, and recalling the characterization of $v_k^{int}$ in \eqref{eq:v_int_def}, we obtain that the boundary integral in \eqref{eq:boundary_integral_v_k} is negative.

{\bf Step 4.} Estimate of the eigenvalue.
Let
\begin{align*}
F_{k} = \left\{ \Phi = \sum_{i=1}^{k} \alpha_{i} v_{i}: \ \alpha=(\alpha_1,\ldots,\alpha_k)\in\R^k \right\} \subset \mathcal{D}^{1,2}_{A_a}(\Omega),
\end{align*}
where $v_i$ are the competitor functions defined in \eqref{eq:v_i_def}.
By \eqref{eq:v_i_int_estimates}, \eqref{eq:v_k_estimate1} we have, for $i\neq j$,
\[
\left| \int_\Omega p(x) v_i\bar{v}_j\dx \right| = \left| \int_\Omega p(x)\varphi_i\bar{\varphi}_j\dx +
\int_{D_{2|a|}^+(0)} p(x) (v_i^{int}\bar{v}^{int}_j-\varphi_i\bar{\varphi}_j)\dx \right| \leq C|a|^4
\]
(the last estimate improves to $|a|^{4+n}$ in case $i=k$ or $j=k$ ).
Therefore $F_k$ is $k$-dimensional subspace of $\mathcal{D}^{1,2}_{A_a}(\Omega)$ for $|a|$ sufficiently small and we have
\[
\lambda_k^a\leq \sup_{\Phi\in F_k} \frac{\|\Phi\|^2_{D^{1,2}_{A_a}(\Omega)}}{\int_\Omega p(x) |\Phi|^2\dx}.
\]
Relation \eqref{eq:v_i_equation} provides
\[
\|\Phi\|^2_{D^{1,2}_{A_a}(\Omega)}= \sum_{i,j=1}^k \alpha_i\alpha_j
\left\{ \lambda_i\int_{\Omega}p(x) \, v_i\bar{v}_j\dx + i \int_{\partial D_{2|a|}^+(0)}(i\nabla+A_a)(v_i^{ext}-v_i^{int})\cdot\nu \bar{v}_j \dsigma \right\}.
\]
Thus we can write
\[
\lambda_k^a 
\leq \sup_{\alpha\in\R^k} \frac{\alpha^T M \alpha}{\alpha^T N \alpha}
= \lambda_{max}(N^{-1}M),
\]
where $\alpha^T$ denotes the transposed of the vector $\alpha$, $\lambda_{max}(\cdot)$ is the largest eigenvalue of a matrix and $M,N$ are $k\times k$ matrices with entries
\[
m_{ij}=\lambda_i\int_{\Omega}p(x) \, v_i\bar{v}_j\dx + i \int_{\partial D_{2|a|}^+(0)}(i\nabla+A_a)(v_i^{ext}-v_i^{int})\cdot\nu \bar{v}_j \dsigma,
\qquad n_{ij}=\int_\Omega p(x) v_i\bar{v}_j \dx.
\]
By exploiting \eqref{eq:v_i_int_estimates}-\eqref{eq:v_k_estimate2}, we see that $M$ has the form in Lemma \ref{lemma:greatesteigenvalue2} and 
\[
N=\begin{pmatrix}
1+O(|a|^{4}) &           &  O( |a|^{4}) & O(|a|^{n +4})  \\
     O(|a|^4)            &  \ddots   &           & \vdots \\
                       &           & 1 + O( |a|^4)  & O(|a|^{n +4}) \\
O(|a|^{n +4})         &   \ldots  & O(|a|^{n +4}) & 1+  O( |a|^{2n +4}) 
\end{pmatrix}.
\]
By writing $N=Id+\mathcal{E}(|a|)$ we have $N^{-1}=\sum_{j=0}^{\infty} (-1)^j \mathcal{E}(|a|)^j \sim Id -\mathcal{E}(|a|)$ as $a\to0$, so that $N^{-1}$ has the same form as $N$.
Therefore $N^{-1}M$ has the same form as $M$ and we can apply  Lemma \ref{lemma:greatesteigenvalue2} obtaining
\[
\lambda_k^a\leq \lambda_k-C_k |a|^{2n+2} +o(|a|^{2n+2}).
\]
The result follows recalling that $2n+2=h$.
\end{proof}

\section{Frequency formula for magnetic eigenfunctions at boundary points}\label{sec:freq_formula}

Throughout this section we assume that $\Omega$ is regular and that $p(x)$ satisfies \eqref{eq:weight_assumptions}. Given a pole $b=(b_1,b_2)\in\mathbb{R}^2$, we recall the following notation from Section \ref{sec:rectified_boundary}
\[
\pi(b)=(0,b_2), \quad D_r^+(\pi(b))=\{ (x_1,x_2)\in\R^2:\, x_1^2+(x_2-b_2)^2<r^2, \, x_1>0 \}.
\]
We define a Almgren-type frequency function in $D_r^+(\pi(b))$ as follows.

\begin{definition} \label{def:EHN}
Let $b\in\C$, $r>0$ such that $b \in D_r^+(\pi(b))$ and $u \in H^1_{A_b}(D_r^+(\pi(b)))$ with $u=0$ on $\{ x_1=0 \}$. Let $\lambda \in\R$, and $p(x)$ satisfy \eqref{eq:weight_assumptions} in $D_r^+(\pi(b))$. We define
\begin{align} \label{eq:E}
E(u,r,\pi(b),\lambda,A_b) = \int_{D_{r}^{+}(\pi(b))} \left( |(i \nabla + A_b) u|^{2} - \lambda p(x)|u|^{2} \right) \, \mathrm{d}x,
\end{align}
\begin{align} \label{eq:H}
H(u,r,\pi(b)) = \frac{1}{r} \int_{\partial D_{r}^{+}(\pi(b))} |u|^{2} \, \mathrm{d} \sigma,
\end{align}
and the frequency function
\begin{align} \label{eq:N}
N(u,r,\pi(b),\lambda,A_b) = \frac{E(u,r,\pi(b),\lambda,A_b)}{H(u,r,\pi(b))}.
\end{align}
\end{definition}

In the notation above we keep track of all the parameters involved, apart from the weight $p$, since we will need to let them change from section to section. The weight is not explicitly mentioned because it does not play a significant role, as long as it satisfies \eqref{eq:weight_assumptions}.

In particular, in this section we will estimate the frequency function for $u=\varphi_k^a$ and $\lambda=\lambda_k^a$. We shall omit the index $k$ since we will work with a fixed $k$ from now up to Subsection \ref{subsec:theorem_no_nodal_lines}. By Lemma \ref{lemma:Riemannmapping}, we can assume that $\partial \Omega$ is locally flat near the origin, so that we consider the following equation
\begin{align} \label{eq:probleminhalfball}
\left\{ \begin{aligned} (i \nabla + A_{a})^{2} \varphi^a & = \lambda^a p(x) \varphi^a  \qquad & D_{2r_0}^{+}(0)  \\
                        \varphi^a & = 0  \qquad & \{x_1 = 0\} .
                        \end{aligned} \right.
\end{align}
Here $r_0$ is chosen such that
\begin{equation}\label{eq:r_0}
r_0 < (2\lambda^a \|p\|_{L^\infty})^{-1/2} \text{ for $|a|$ sufficiently small},
\end{equation}
which is possible due to the fact that $p$ is bounded and that $\lambda^a\to\lambda$ as $a\to0$, as recalled in Theorem \ref{theorem:continuity}. For $r<r_0$ and $|a|<r$ we have that $D_r^+(\pi(a))\subset D_{2r_0}^+(0)$ so that, for such $r$ and $a$, the frequency function for solutions of \eqref{eq:probleminhalfball} is well defined.

\subsection{Estimates on $H(\varphi^a,r,\pi(a))$}

We can compute the derivative of $H$ with respect to $r$ similarly to the standard frequency function for non-magnetic eigenfunctions.
\begin{lemma} \label{lemma:derivativeofH}
If $\varphi^a$ is a solution of \eqref{eq:probleminhalfball}, for $a_1 < r<r_0$ we have
\begin{align}\label{eq:derivative_H}
\frac{\mathrm{d}}{\mathrm{d}r} H(\varphi^a,r,\pi(a)) = -\frac{2 i}{r} \int_{\partial D_{r}^{+}(\pi(a))} (i \nabla + A_a) \varphi^a \cdot \nu \overline{\varphi^a} \, \mathrm{d} \sigma = \frac{2}{r} E(\varphi^a,r,\pi(a),\lambda^a, A_a).
\end{align} 
\end{lemma}
\begin{proof}
By the change of variables $y = (x- \pi(a))/r$ we have
\begin{align*}
H(\varphi^a,r,\pi(a))=\frac{1}{r} \int_{\partial D_{r}^{+}(\pi(a))} |\varphi^a|^{2}(x) \, \mathrm{d}\sigma(x) = \int_{\partial D_{1}^{+}(0)} |\varphi^a|^{2}(ry + \pi(a)) \, \mathrm{d}\sigma(y).
\end{align*}
By taking the derivative with respect to $r$ we obtain
\begin{align*}
\frac{\mathrm{d}}{\mathrm{d}r} H(\varphi^a,r,\pi(a)) & = 2 Re \int_{\partial D_{1}^{+}(0)} \nabla \varphi^a (ry+\pi(a)) \cdot y\, \overline{\varphi^a(ry+\pi(a))} \, \mathrm{d} \sigma(y) \\
& = \frac{2}{r} Re \int_{\partial D_{r}^{+}(\pi(a))} \nabla \varphi^a \cdot \nu \, \overline{\varphi^a} \, \mathrm{d}\sigma(x) \\
 & = \frac{2}{r} Re \left\{-i\int_{\partial D_r^+(\pi(a))} (i\nabla + A_a) \varphi^a \cdot \nu \overline{\varphi^a} \, \mathrm{d}\sigma\right\},
\end{align*}
where we used the fact that $Re(-i|\varphi^a|^2 A_a\cdot\nu)=0$.
On the other hand, by testing equation \eqref{eq:probleminhalfball} by $\varphi^a$ and integrating by parts, we see that
\begin{align} \label{eq:newE}
E(\varphi^a, r,\pi(a), \lambda^a, A_a) = - i \int_{\partial D_r^+(\pi(a))} (i \nabla + A_a) \varphi^a \cdot \nu \, \overline{\varphi^a}\dsigma \quad \in\R,
\end{align}
which concludes the proof.
\end{proof}

We can prove the following estimate.
\begin{lemma} \label{lemma:inferiorboundH}
Let $\varphi^a$ be a solution of \eqref{eq:probleminhalfball} and $r_0$ be as in \eqref{eq:r_0}.
If $a_1 < r_1 < r_2 < r_{0}$ then
\begin{align} \label{eq:inferiorboundH}
\frac{H(\varphi^a,r_2,\pi(a))}{H(\varphi^a,r_1,\pi(a))} \geq e^{-C r_0^2}\ \left(\frac{r_2}{r_1}\right)^{2}.
\end{align}
If $|a|$ is sufficiently small, we can choose $C=4 \lambda \|p\|_{L^\infty}$, where $\lambda$ is the limit of $\lambda^a$ as $a\to0$. 
\end{lemma}
\begin{proof}
By combining Lemmas \ref{lemma:Poincaré} and \ref{lemma:inequality2} we see that
\begin{align} \label{eq:inequality3}
 \frac{1}{r^{2}} \int_{D_{r}^{+}(\pi(a))} |\varphi^a|^{2}\dx \leq 2 \int_{D_{r}^{+}(\pi(a))} |(i\nabla +A_a) \varphi^a |^{2}\dx.
\end{align}
Next we apply Lemma \ref{lemma:derivativeofH} and, in order, the inequalities \eqref{eq:inequality3} and \eqref{eq:inequality2} in the following way
\begin{align*}
\frac{\mathrm{d}}{\mathrm{d}r}H(\varphi^a,r,\pi(a)) & 
= \frac{2}{r} \int_{D_{r}^{+}(\pi(a))}(|(i \nabla + A_a) \varphi^a |^{2} - \lambda^a p(x) |\varphi^a|^{2})\dx \\
          & \geq \frac{2}{r} (1 - 2 \lambda^a \|p\|_{L^\infty} r^{2}) \int_{D_{r}^{+}(\pi(a))} |(i \nabla + A_a) \varphi^a|^{2} dx \\
          & \geq \frac{2}{r} (1 - 2 \lambda^a \|p\|_{L^\infty} r^{2}) H(\varphi^a,r,\pi(a))).
\end{align*}
Integrating the last inequality between $r_1$ and $r_2$ we obtain 
\begin{align*}
\log \left(\frac{H(\varphi^a,r_2,\pi(a))}{H(\varphi^a,r_1,\pi(a))}\right) & \geq \log \left( \frac{r_2}{r_1}\right)^2 - 2 \lambda^a \|p\|_{L^\infty} ( r_2^2 - r_1^2) \\
         & \geq \log \left( \frac{r_2}{r_1}\right)^2 - 2 \lambda^a \|p\|_{L^\infty} r_0^2.
\end{align*}
Taking the exponential of both sides and recalling that $\lambda^a\to\lambda$ we obtain the statement. 
\end{proof}

\begin{remark}\label{rem:Egeq0}
The previous proof shows that $E(\varphi^a,r,\pi(a),\lambda^a,A_a)\geq0$ for $r<r_0$.
\end{remark}

\subsection{Estimates on $E(\varphi^a,r,\pi(a),\lambda^a,A_a)$}

We will need the following Pohozaev-type identity for the solution $\varphi^a$ of \eqref{eq:probleminhalfball}. Also compare with \cite[Section 4]{FFT2011}.

\begin{lemma}[Pohozaev-type identity] \label{lemma:pohozaev}
If $\varphi^a$ is the solution of \eqref{eq:probleminhalfball}, the following identity is valid for $a_1 < r <r_0$
\begin{equation} \label{eq:pohozaev}
\begin{split}
& \frac{r}{2} \int_{\partial D_{r}^{+}(\pi(a))}\left\{ | (i \nabla + A_a) \varphi^a |^{2} - 2 | ( i \nabla + A_a) \varphi^a \cdot \nu |^{2} - \lambda^a p |\varphi^a|^{2} \right\}\dsigma \\
& + \lambda^a \int_{D_{r}^{+}(\pi(a))} |\varphi^a|^{2} \left(p + \frac{\nabla p\cdot (x-\pi(a))}{2}\right) \dx 
+ M_a   = 0,
\end{split}
\end{equation}
where
\begin{equation}\label{eq:M_a_def}
\begin{split}
M_a=\lim_{\eps\to0} \int_{\partial D_\eps(a)} \left\{ Re \left[  (i \nabla + A_a) \varphi^a \cdot \nu \, \overline{ ( i \nabla + A_a) \varphi^a \cdot (x-\pi(a)) } \right] \right.\\
\left. - \frac{1}{2} | ( i \nabla + A_a) \varphi^a |^2 (x - \pi(a)) \cdot \nu \right\} \dsigma.
\end{split}
\end{equation}
\end{lemma}
\begin{proof}
We test the equation \eqref{eq:probleminhalfball} with the vector field $\xi=(i\nabla+A_a)\varphi^a\cdot (x-\pi(a))$ in $D_r^+(\pi(a))\setminus D_\eps(a)$. We need to remove a small ball around the singularity because $\nabla\varphi^a$ may be singular at $a$ (it is singular in the case that $\varphi^a$ has a zero of order $1/2$ at $a$). Multiplying by $i$ and taking the real part we obtain
\begin{align}\label{eq:xi_pohozaev}
Re \left\{ i \int_{D_{r}^{+}(\pi(a)) \backslash D_{\varepsilon}(a)} (i \nabla + A_a)^2 \varphi^a \, \overline{ \xi }\dx \right\} = Re \left\{ i \lambda^a \int_{D_{r}^{+}(\pi(a))  \backslash D_{\varepsilon}(a)} \varphi^a p(x) \, \overline{ \xi }\dx \right\}.
\end{align}
Similarly to \eqref{eq:divergence_u_square_x}, the following identity with the weight holds
\begin{align*}
\frac{1}{2} \nabla \cdot \left(p |\varphi^a|^2 (x-\pi(a)) \right)=
|\varphi^a|^2 \left(p+ \frac{\nabla p\cdot (x-\pi(a))}{2} \right) + Re \left( i p \varphi^a \overline{\xi} \right).
\end{align*}
It allows to rewrite the right hand side of \eqref{eq:xi_pohozaev} as
\[
\frac{\lambda^a}{2} \int_{\partial (D_{r}^{+}(\pi(a))\backslash D_{\varepsilon}(a))} p |\varphi^a|^{2} \, (x-\pi(a)) \cdot \nu \dsigma - \lambda^a \int_{D_{r}^{+}(\pi(a)) \backslash D_{\varepsilon}(a)} |\varphi^a| ^{2} \left(p + \frac{\nabla p \cdot (x-\pi(a))}{2} \right) \dx.
\]
Taking the limit as $\varepsilon \to 0$, we obtain
\[
\frac{\lambda^a r}{2} \int_{\partial D_{r}^{+}(\pi(a))} p |\varphi^a|^{2} \dsigma - \lambda^a \int_{D_{r}^{+}(\pi(a))} |\varphi^a| ^{2} \left(p + \frac{\nabla p \cdot (x-\pi(a))}{2} \right) \dx.
\]
The integral on $\partial D_{\varepsilon}(a)$ vanishes as $\varepsilon \to 0$ because $|\varphi^a|$ behaves at least like $\varepsilon^{1/2}$ on $\partial D_\eps(a)$ by \eqref{eq:varphia_asymptotic_expansion}.
Next we look at the left-hand side of \eqref{eq:xi_pohozaev}. Integrating by parts and using the identity 
\begin{align*}
 Re \left( i (i \nabla + A_a) \varphi^a \cdot \overline{ (i \nabla + A_a) \xi } \right) = \frac{1}{2} \nabla \cdot ( | (i \nabla + A_a) \varphi^a|^{2} \, (x-\pi(a))) \qquad \text{a.e. }x\in D_r^+,
\end{align*}
we rewrite it as
\begin{align*}
& Re \left\{ i \int_{D_{r}^{+}(\pi(a))\backslash D_{\varepsilon}(a)} (i \nabla + A_a) \varphi^a \cdot \overline{ (i \nabla + A_a) \xi }\dx - \int_{\partial(D_{r}^{+}(\pi(a))\backslash D_{\varepsilon}(a))} (i \nabla + A_a) \varphi^a \cdot \nu \, \overline{ \xi }\dsigma \right\} \\
& = \int_{\partial(D_{r}^{+}(\pi(a))\backslash D_{\varepsilon}(a))} \left\{ \frac{1}{2} |(i \nabla + A_a) \varphi^a |^{2}(x-\pi(a))\cdot\nu - Re[ (i \nabla + A_a) \varphi^a \cdot \nu \, \overline{ \xi }] \right\} \dsigma \\
& = r \int_{\partial D_{r}^{+}(\pi(a))} \left\{\frac{1}{2} |(i \nabla + A_a) \varphi^a |^{2} - | (i \nabla + A_a) \varphi^a \cdot \nu |^{2}\right\}\dsigma \\
&+ \int_{\partial D_\eps(a)} \left\{ Re[(i\nabla +A_a)\varphi^a\cdot\nu\overline{\xi}] -\frac{1}{2}|(i\nabla+A_a)\varphi^a|^2(x- \pi(a))\cdot\nu  \right\}\dsigma.
\end{align*}
By taking the limit as $\eps\to0$ and by combining the two contributions in \eqref{eq:xi_pohozaev} we obtain the result.
\end{proof}

This identity allows to compute the derivative of $E(\varphi^a,r,\pi(a),\lambda^a,A_a)$ with respect to $r$.

\begin{lemma} \label{lemma:derivativeofE}
If $\varphi^a$ is a solution of \eqref{eq:probleminhalfball}, then for $a_1 < r<r_0$ we have
\begin{equation}\label{eq:derivative_E}
\begin{split}
\frac{\mathrm{d}}{\mathrm{d}r} E(\varphi^a,r,\pi(a),\lambda^a,A_a) & = 
2 \int_{\partial D_{r}^{+}(\pi(a))} | (i \nabla + A_a) \varphi^a \cdot \nu |^{2} \, \mathrm{d} \sigma \\
&-  \frac{\lambda^a}{r} \int_{D_{r}^{+}(\pi(a))} |\varphi^a|^{2} (2p+\nabla p\cdot (x-\pi(a)))\, \mathrm{d}x 
- \frac{2}{r} M_a ,
\end{split}
\end{equation}
where $M_a$ is defined in \eqref{eq:M_a_def}.
\end{lemma}
\begin{proof}
We have
\begin{align*}
\frac{\mathrm{d}}{\mathrm{d}r} E(\varphi^a,r,\pi(a),\lambda^a,A_a) & = \int_{\partial D_{r}^{+}(\pi(a))} \left( |(i \nabla + A_a) \varphi^a|^{2} - \lambda^a p |\varphi^a|^{2} \right) \, \mathrm{d}\sigma.
\end{align*}
Then we use the Pohozaev identity \eqref{eq:pohozaev} to conclude.
\end{proof}

In what follows we will estimate the remainder $M_a$ which appears in the derivative of $E$ in equation \eqref{eq:derivative_E}.

\begin{lemma}\label{lemma:M_a_rewritten}
Let $v(y)= e^{-i\theta(y)} \varphi^a(a_1y^2+a)$, defined in the set $\{ y:\ a_1y^2+a\in D_{2r_0}^+(0) \}$. Then
\begin{equation}\label{eq:M_a}
M_a=\pi Re\left\{\left(\frac{\partial v(0)}{\partial y}\right)^2\right\}.
\end{equation}
\end{lemma}
\begin{proof}
First we shall prove that
\begin{align}\label{eq:M_a_rewritten}
M_a = \lim_{\varepsilon \to 0} \int_{\partial D_{\varepsilon}(a)} Re[ (i \nabla + A_a) \varphi^a \cdot \nu \, \overline{(i \nabla + A_a) \varphi^a \cdot (x-\pi(a))}] \dsigma = 
\frac{a_1 \pi}{4}(c_1^2-d_1^2),
\end{align}
where $c_1=c_1(a)$, $d_1=d_1(a)$ are the coefficients appearing in the asymptotic expansion \eqref{eq:varphia_asymptotic_expansion} of $\varphi^a$, and $a=(a_1,a_2)$.
Indeed, by differentiating \eqref{eq:varphia_asymptotic_expansion} we obtain
\[
(i \nabla + A_a) \varphi^a=\frac{i e^{i \theta_a /2}}{2\sqrt{r_a}}\left(c_1\cos\frac{\theta_a}{2}- d_1\sin\frac{\theta_a}{2}, c_1\sin\frac{\theta_a}{2}+d_1\cos\frac{\theta_a}{2}\right)+ o(r_a^{-1/2}) \quad\text{as } r_a\to0,
\]
and hence
\begin{align*}
|(i \nabla + A_a) \varphi^a|^{2} = \frac{1}{4 r_a}  (c_{1}^{2} + d_{1}^{2}) + o(r_a^{-1}) \quad\text{as } r_a\to0 .
\end{align*}
Moreover notice that $x-\pi(a) = (a_1,0) + \varepsilon (\cos\theta_a, \sin\theta_a)$ and $\nu = (\cos\theta_a, \sin\theta_a)$ on $\partial D_{\varepsilon}(a)$.
Therefore
\begin{align*}
\lim_{\varepsilon \to 0} \int_{\partial D_{\varepsilon}(a)} |(i \nabla + A_a) \varphi^a|^{2} (x-\pi(a)) \cdot \nu \dsigma =\\
\lim_{\varepsilon \to 0} \varepsilon \int_{0}^{2 \pi} \left[ \frac{1}{4 \varepsilon}(c_{1}^{2} + d_{1}^{2})  + o(\eps^{-1}) \right] \left[a_{1} \cos\theta_a  + \varepsilon \right] \, \mathrm{d}\theta_a
= 0,
\end{align*}
and we have estimated the second term in \eqref{eq:M_a_def}.
By a direct calculation one estimates the first term in \eqref{eq:M_a_def} and obtains \eqref{eq:M_a_rewritten}.

Now, by changing variables in \eqref{eq:varphia_asymptotic_expansion}, we obtain the following expansion for $v$
\[
v(r,\theta)= \sqrt{a_1} r \left(c_1\cos\theta+ d_1\sin\theta\right)+o(r)\quad\text{as } r\to0.
\]
Hence we have $\frac{\partial v(0)}{\partial y}=\frac{\sqrt{a_1}}{2}\left(c_1-i d_1\right)$ and \eqref{eq:M_a} follows by combining with \eqref{eq:M_a_rewritten}.
\end{proof}
From \eqref{eq:M_a_rewritten}, we remark that the constant $M_a$ is identically zero if the eigenfunction $\varphi^a$ has a zero of order stricly greater than $1$ at $a$.
\begin{lemma}\label{lemma:M_a_bound}
There exists $C>0$ not depending on $a_1$ such that
\[
\frac{|M_a|}{H(\varphi^a,2a_1,\pi(a))}\leq C.
\]
\end{lemma}
\begin{proof}
The quantity $M_a$ is expressed in terms of $v$ in Lemma \ref{lemma:M_a_rewritten}. We also rewrite
\begin{equation}\label{eq:H(u,2a)}
H(\varphi^a,2a_1,\pi(a))=\int_\gamma v^2|y|\dsigma,
\end{equation}
where, letting $\Omega=\{y: \ a_1y^2+a\in  D_{2a_1}^+(\pi(a))\}$, we have $\gamma = \partial \Omega$.
By Lemma \ref{lemma:gauge_invariance}, $v$ solves $-\Delta v=4a_1^2|y|^2\tilde{p}\lambda^a v$ in $\Omega$, where $\tilde{p}(y) = p(a_1 y^2 + a)$ has the same properties as $p(x)$. Since $\gamma$ does not depend on $a$, Lemma \ref{lemma:appendix_green} applies, providing for $a_1$ sufficiently small the representation formula
\[
v(x)=-\int_\gamma v(y)\, \partial_{\nu}G(x,y) \,\textrm{d}\sigma(y),
\]
for $x\in\Omega$ and moreover, 
\[
\|\partial_{x_i}G(x,\cdot)-\partial_{x_i}\Phi(x,\cdot)\|_{W^{2,q}(\Omega)}\leq Ca_1^2 \quad \text{ for } 1 \leq q < 2.
\]
Therefore we have, by the H\"older and traces inequalities (for the trace embedding, see [\cite{AdamsFournierBook2003},Theorem 5.36]) and the estimate above, taking for example $q = 4/3$, we have
\[
|\partial_{x_i} v(0)|^2 
=\left( \int_\gamma v \partial_{\nu}\partial_{x_i}G(0,y) \,\textrm{d}\sigma(y)\right)^2
\leq \int_\gamma v^2 \dsigma \ \int_\gamma |\partial_\nu\partial_{x_i}G(0,y)|^2 \dy  \leq C \int_\gamma v^2|y|\dsigma.
\]
Hence, by Lemma \ref{lemma:M_a_rewritten} and by \eqref{eq:H(u,2a)}, we have
\[
\frac{|M_a|}{H(\varphi^a,2a_1,\pi(a))}\leq C \frac{|\nabla v(0)|^2}{H(\varphi^a,2a_1,\pi(a))} \leq C. \qedhere
\]
\end{proof}

\begin{lemma}\label{lemma:M_a/H}
There exists $C > 0$ independent of $a_1$ such that
\begin{align*}
\frac{|M_a|}{H(\varphi^a,ka_1,\pi(a))} \leq \frac{C}{k^{2}}
\text{ for every }  k>2.
\end{align*}
\end{lemma}
\begin{proof}
It is a straightforward consequence of Lemmas \ref{lemma:inferiorboundH} and \ref{lemma:M_a_bound}:
\[
\frac{|M_a|}{H(\varphi^a,ka_1,\pi(a))} =  \frac{|M_a|}{H(\varphi^a,2a_1,\pi(a))}\cdot \frac{H(\varphi^a,2a_1,\pi(a))}{H(\varphi^a,ka_1,\pi(a))} \leq \frac{C}{k^{2}}. \qedhere
\] 
\end{proof}

\subsection{Estimates on $N(\varphi^a,r,\pi(a),\lambda^a,A_a)$}
The function $N$ may not be increasing, because of the remainder $M_a$ which appears in the derivative of $E$ in \eqref{eq:derivative_E}. Nonetheless, we can use the estimates proved in the previous paragraph to obtain a bound from below on the derivative of $N$.

\begin{lemma} \label{lemma:bound1}
Let $\varphi^a$ be a solution of \eqref{eq:probleminhalfball} and $r_0$ be as in \eqref{eq:r_0}.
For $a_1 < r < r_{0}$ we have
\begin{align} \label{eq:bound1}
\frac{1}{r^{2}} \int_{D_{r}^{+}(\pi(a))} |\varphi^a|^{2} \, \mathrm{d}x \leq  
\frac{E(\varphi^a,r,\pi(a),\lambda^a,A_a)+H(\varphi^a,r,\pi(a))}{1-C r_0^2}.
\end{align}
If $|a|$ is sufficiently small, we can choose $C=2\lambda\|p\|_{L^\infty}$, where $\lambda$ is the limit of $\lambda^a$ as $a\to 0$. 
\end{lemma}
\begin{proof}
On the one hand, the Poincaré inequality \eqref{eq:Poincaré} provides
\begin{align*}
\frac{1}{r^{2}} \int_{D_{r}^{+}(\pi(a))} |\varphi^a|^{2} \, \mathrm{d}x - \lambda^a \int_{D_{r}^{+}(\pi(a))} p(x) |\varphi^a|^{2} \, \mathrm{d}x \leq E(\varphi^a,r,\pi(a),\lambda^a,A_a) + H(\varphi^a,r,\pi(a)).
\end{align*}
On the other hand, if we take $r < r_{0} $, we obtain that
\begin{align*}
\frac{1-r_0^2\lambda^a\|p\|_{L^\infty}}{r^{2}} \int_{D_{r}^{+}(\pi(a))} |\varphi^a|^{2} \, \mathrm{d}x \leq \frac{1}{r^{2}} \int_{D_{r}^{+}(\pi(a))} |\varphi^a|^{2} \, \mathrm{d}x - \lambda^a \int_{D_{r}^{+}(\pi(a))} p(x) |\varphi^a|^{2} \, \mathrm{d}x.
\end{align*}
The result follows by combining the previous two inequalities.
\end{proof}

\begin{lemma} \label{lemma:boundN1}
Let $\varphi^a$ be a solution of \eqref{eq:probleminhalfball} and $r_0$ be as in \eqref{eq:r_0}. For every $k>1$, $a_1<r_0/k$ and $ka_1<r<r_0$ we have
\begin{align}  \label{eq:boundN1}
N(\varphi^a,r,\pi(a),\lambda^a,A_a) \leq [N(\varphi^a,r_0,\pi(a),\lambda^a,A_a)+1] \exp\left(\frac{Ce^{Cr_0^2}}{k^2}+\frac{C r_0^2}{1-C r_0^2}\right)-1,
\end{align}
with $C>0$ independent from $a_1,k,r,r_0$.
\end{lemma}
\begin{proof}
Let for the moment $N=N(\varphi^a,r,\pi(a),\lambda^a,A_a)$ and analogously for $H$ and $E$.
We use Lemmas \ref{lemma:derivativeofH} and \ref{lemma:derivativeofE} to obtain, for $r>a_1$,
\begin{align}\label{eq:N_derivative}
\frac{\mathrm{d} N}{\mathrm{d}r} = \frac{1}{H^{2}} \left\{ \frac{2}{r} \int_{\partial D_{r}^{+}(\pi(a))} |(i \nabla + A_a) \varphi^a \cdot \nu|^2 \, \mathrm{d}\sigma \int_{\partial D_{r}^{+}(\pi(a))} |\varphi^a|^{2} \, \mathrm{d}\sigma \right. \notag\\
\left. - \frac{2}{r} \left( i \int_{\partial D_{r}^{+}(\pi(a))} (i \nabla + A_a) \varphi^a \cdot \nu \, \overline{\varphi^a} \, \mathrm{d}\sigma \right)^{2}\right\} \notag \\
              - \frac{1}{r^{2}H^{2}} \left\{2M_a + \lambda^a \int_{D_{r}^{+}(\pi(a))} |\varphi^a|^{2}(2p+\nabla p\cdot (x-\pi(a))) \, \mathrm{d}x \right\} \int_{\partial D_{r}^{+}(\pi(a))} |\varphi^a|^{2} \, \mathrm{d} \sigma  \notag \\
              \geq - \frac{1}{rH} \left\{2 |M_a| + \lambda^a \|2p+\nabla p\cdot (x-\pi(a))\|_{L^\infty} \int_{D_{r}^{+}(\pi(a))} |\varphi^a|^{2} \, \mathrm{d}x \right\}.
\end{align}
In the last step we used the Schwarz inequality and the regularity assumption on $p$ \eqref{eq:weight_assumptions}. Therefore we have
\begin{equation}\label{eq:logN_derivative}
\frac{\mathrm{d}}{\mathrm{d}r}\log(N+1)\geq 
- \frac{1}{r(E+H)} \left\{2 |M_a| + \lambda^a \|2p+\nabla p\cdot (x- \pi(a))\|_{L^\infty} \int_{D_{r}^{+}(\pi(a))} |\varphi^a|^{2} \, \mathrm{d}x \right\}.
\end{equation}
We look at the first term in the right hand side of \eqref{eq:logN_derivative}. By Lemma \ref{lemma:inferiorboundH} we have
\[
\frac{H(\varphi^a,r,\pi(a))}{H(\varphi^a,ka_1, \pi(a))}\geq e^{-C r_0^2}\ \frac{r^2}{(ka_1)^2} \, ,
\qquad ka_1<r<r_0.
\]
This together with Remark \ref{rem:Egeq0} and Lemma \ref{lemma:M_a/H} provides
\begin{align*}
- \frac{|M_a|}{r(E+H)} \geq - \frac{|M_a|}{rH} \geq -\frac{e^{C r_0^2}}{r^3}\ \frac{|M_a|(ka_1)^2}{H(\varphi^a,ka_1,\pi(a))}
\geq - C e^{Cr_0^2}\frac{a_1^2}{r^3}, \qquad ka_1<r<r_0.
\end{align*}
Concerning the second term in the right hand side of \eqref{eq:logN_derivative}, we apply Lemma \ref{lemma:bound1} to obtain
\begin{align*}
- \frac{1}{r(E+H)} \int_{D_{r}^{+}(\pi(a))} |\varphi^a|^{2} \dx
\geq - \frac{r}{1-C r_0^2}, \qquad ka_1<r<r_0.
\end{align*}
Thus we have obtained
\[
\frac{\mathrm{d}}{\mathrm{d}r}\log(N+1)\geq -\frac{Ce^{Cr_0^2}a_1^2}{r^3}-\frac{C r}{1-C r_0^2},
\qquad ka_1<r<r_0.
\]
By integrating between $r$ and $r_0$ we arrive at
\begin{align*}
\log \frac{N(\varphi^a,r_0,\pi(a),\lambda^a,A_a)+1}{N(\varphi^a,r,\pi(a),\lambda^a,A_a)+1} 
& \geq C a_1^{2} \left(\frac{e^{Cr_0^2}}{r_{0}^{2}} - \frac{e^{Cr_0^2}}{r^{2}} \right) - \frac{C}{1- C r_0^2} (r_{0}^2 - r^2) \\
& \geq - Ce^{Cr_0^2}\frac{a_1^{2}}{r^2} - \frac{C r_0^2}{1-C r_0^2} 
\geq - \frac{C e^{Cr_0^2}}{k^2} - \frac{C r_0^2}{1-C r_0^2},
\end{align*}
for  $ka_1<r<r_0$. The statement follows by exponentiating both terms.
\end{proof}

\section{Proof of Proposition \ref{prop:unique_limit_profile}}\label{sec:prop_limit_profile}

\begin{proof}[Proof of Proposition \ref{prop:unique_limit_profile}]

{\bf Step 1.} Suppose by contradiction that there are two solutions $\psi$ and $\tilde{\psi}$ to \eqref{eq:limit_profile}, which do not differ by a multiplicative constant. 
By Proposition \ref{proposition:asymptotic_expansion_eigenfunction} we have
\[
\psi(|x-e|,\theta_e)=e^{i\frac{\theta_e}{2}} \sqrt{|x-e|} \left( c_1 \cos\frac{\theta_e}{2}+d_1\sin\frac{\theta_e}{2} \right) +o(\sqrt{|x-e|}),
\]
\[
\tilde{\psi}(|x-e|,\theta_e)=e^{i\frac{\theta_e}{2}} \sqrt{|x-e|} \left( \tilde{c}_1 \cos\frac{\theta_e}{2}+\tilde{d}_1\sin\frac{\theta_e}{2} \right) +o(\sqrt{|x-e|}),
\]
as $|x-e|\to0$, $e=(1,0)$. Suppose first that $c_1^2+d_1^2\neq0$. We consider the linear combination $t\psi+\tilde{\psi}$ that we can write, thanks to Proposition \ref{proposition:asymptotic_expansion_eigenfunction} and the expressions above,
\begin{align}\label{eq:t_psi_asymptotic}
(t \psi + \tilde{\psi}) (|x-e|,\theta_e) = e^{i \frac{\theta_e}{2}} \sqrt{|x-e|} \left[ (t c_1 + \tilde{c}_1) \cos \frac{\theta_e}{2} + (t d_1 + \tilde{d}_1) \sin \frac{\theta_e}{2} \right] + o(\sqrt{|x-e|}). 
\end{align}
The parameter $t$ is chosen in such a way that
\[
M
= \pi \frac{(tc_1+\tilde{c}_1)^2-(td_1+\tilde{d}_1)^2}{4}=0,
\]
where $M$ is the constant associated to $t \psi + \tilde{\psi}$, see \eqref{eq:M_a_rewritten}. More explicitly we have $t=(\tilde{d}_1-\tilde{c}_1)/(c_1-d_1)$ if $c_1\neq d_1$ and $t=-(\tilde{c}_1+\tilde{d}_1)/(c_1+d_1)$ if $c_1 \neq -d_1$.
Exactly as in \eqref{eq:N_derivative} we have for $r > 1$
\begin{equation}\label{eq:N_increasing}
\frac{\mathrm{d}}{\mathrm{d}r} N(t\psi+\tilde{\psi},r,0,0,A_e) 
\geq -\frac{2M}{r H(t\psi+\tilde{\psi},r,0)} =0,
\end{equation}
thanks to our choice of $t$, so that $N(t\psi+\tilde{\psi},\cdot,0 ,0,A_e)$ is increasing. 

We claim that
\begin{equation}\label{eq:linear_combination_psi_N}
\lim_{r\to\infty} N(t\psi+\tilde{\psi},r,0,0,A_e) \leq 1.
\end{equation}
Suppose by contradiction that there exist $\delta,R_\delta>0$ such that $N(t\psi+\tilde{\psi},r,0,0,A_e) \geq 1+\delta$ for every $r>R_\delta$. Then,  since $t\psi+\tilde\psi$ solves the equation, proceeding as in \eqref{eq:derivative_H}, we find
\begin{equation}\label{eq:tilde_psi_aux}
\frac{\mathrm{d}}{\mathrm{d}r}\log H(t\psi+\tilde{\psi},r,0)
=\frac{2}{r} N(t\psi+\tilde{\psi},r,0,0,A_e) \geq \frac{2}{r}(1+\delta),
\qquad r>R_\delta.
\end{equation}
Integrating between $R_\delta$ and $r$ we obtain
\begin{equation}\label{eq:linear_combination_psi_H}
H(t\psi+\tilde{\psi},r,0) \geq Cr^{2(1+\delta)}, \qquad r>R_\delta.
\end{equation}
On the other hand, by assumption \eqref{eq:limit_profile_normalization} we have (by eventually taking a larger $R_\delta$)
\[
N(\psi,r,0,0,A_e)<1+\frac{\delta}{2},\quad N(\tilde{\psi},r,0,0,A_e)<1+\frac{\delta}{2} \qquad r>R_\delta.
\]
Proceeding as above, this implies $H(\psi,r,0)+H(\tilde{\psi},r,0) \leq Cr^{2(1+\delta/2)}$ for $r>R_\delta$. Hence, by the Young inequality, we obtain
\[
H(t\psi+\tilde{\psi},r,0) \leq 2\left[ H(t\psi,r,0)+H(\tilde{\psi},r,0) \right]
\leq C r^{2(1+\delta/2)}, \qquad r>R_\delta,
\]
which contradicts \eqref{eq:linear_combination_psi_H}. Hence \eqref{eq:linear_combination_psi_N} is proved.

On the other hand, it is not difficult to see that, since $t\psi+\tilde{\psi}$ vanishes on $\{x_1=0\}$ but is not identically zero, we must have 
\begin{equation}\label{eq:Ngeq1}
N(t\psi+\tilde{\psi},0,0,0,A_e)\geq1.
\end{equation} 
Indeed, suppose by contradiction that there exist $\eps>0$, $r_\eps<1$ such that
\[
N(t\psi+\tilde{\psi},r,0,0,A_e)<1-\eps, \quad r<r_\eps.
\]
Using this inequality as we did in \eqref{eq:tilde_psi_aux}, and then integrating between $r$ and $r_\eps$, we arrive at
\[
\frac{H(t\psi+\tilde\psi,r_\eps,0)}{H(t\psi+\tilde\psi,r,0)} \leq \left(\frac{r_\eps}{r}\right)^{2-2\eps}.
\]
Conversely, Lemmas \ref{lemma:derivativeofH} and \ref{lemma:inequality2} provide
\[
\frac{\mathrm{d}}{\mathrm{d}r} H(t\psi+\tilde\psi,r,0)
=\frac{2}{r} \int_{D_r^+(0)} |(i\nabla+A_e)(t\psi+\tilde\psi)|^2 \,dx 
\geq \frac{2}{r} H(t\psi+\tilde\psi,r,0),
\]
and hence
\[
\frac{H(t\psi+\tilde\psi,r_\eps,0)}{H(t\psi+\tilde\psi,r,0)} \geq \left(\frac{r_\eps}{r}\right)^2,
\]
which contradicts the previous inequality for $r<r_\eps$.
 
We conclude from \eqref{eq:N_increasing}, \eqref{eq:linear_combination_psi_N} and \eqref{eq:Ngeq1} that $N(t\psi+\tilde{\psi},r,0,0,A_e) \equiv 1$ and, in turn, that $t\psi+\tilde{\psi}=e^{i\theta_e/2} r g(\theta)$, for some function $g$ depending only on the angle. This contradicts the asymptotic behavior \eqref{eq:t_psi_asymptotic} of $t\psi+\tilde{\psi}$ at $e$. We have obtained uniqueness up to a multiplicative constant in case $c_1^2+d_1^2\neq0$. If $c_1=d_1=0$ then all the previous considerations apply with $\tilde{\psi}$ in place of $t\psi+\tilde{\psi}$ and we still obtain a contradiction.

{\bf Step 2.} We will use some ideas in \cite{FelliTerracini2013}, in particular Lemmas 2.4 and 2.9 (see also \cite{AbatangeloFelliTerracini2013}). Let $Q_1=\{(x_1,x_2)\in\R^2: \, x_1>0,\, x_2>0\}$ and let $\Gamma_1=\{ (x_1,0)\in\R^2: \, 0<x_1<1 \}$. We consider the following minimization problem
\begin{equation}\label{eq:beta_def}
\frac{\beta}{2}=\inf\left\{ \int_{Q_1}|\nabla u|^2\dx: \, u\in \mathcal{D}^{1,2}(Q_1), \, u=0 \text{ on } \{x_1=0\}, \, u=-x_1 \text{ on } \Gamma_1 \right\},
\end{equation}
where we denote by $\mathcal{D}^{1,2}(Q_1)$ the closure of $C_0^\infty(Q_1)$ with respect to $\|\nabla u\|_{L^2(Q_1)}$. By standard variational methods the infimum is achieved by a unique function $w\in \mathcal{D}^{1,2}(Q_1)$ (see for example \cite[Theorem 8.4]{SalsaBook}). Due to the symmetries of the problem, we can extend $w$ to $\R^2_+$ in such a way that $w(x_1,-x_2)=w(x_1,x_2)$ and moreover $w$ satisfies the following properties
\[
- \Delta w = 0 \text{ in } \R^2_+\setminus \Gamma_1, \quad
w = 0 \text{ on } \{x_1=0\}, \quad
w = -x_1 \text{ on } \Gamma_1, \quad
\int_{\R^2_+}|\nabla w|^2\dx=\beta <\infty.
\]
By the maximum principle we can suppose $w<0$ in $\R^2_+$. One can check that $\tilde{\psi}=e^{i\theta_e/2}(x_1+w)$ solves \eqref{eq:limit_profile}, by passing to the double covering as in Lemma \ref{lemma:gauge_invariance}. We aim at showing that $\psi=\tilde{\psi}$; thanks to step 1, it will be sufficient to prove that $\tilde{\psi}$ satisfies \eqref{eq:limit_profile_normalization}. Let $\tilde{w}$ be the Kelvin transform of $w$, that is $\tilde{w}(y)=w(y/|y|^2)$ for $|y|<1$. Because $w$ is harmonic outside of $D_1^+(0)$, $\tilde{w}$ is harmonic in $D_1^+(0)$ with zero boundary conditions on $\{y_1=0\}$. Moreover, $\int_{D_1^+(0)} |\nabla \tilde{w}|^2\dx= \int_{\mathbb{R}^2_+ \backslash D_1^+(0)} |\nabla w|^2 \dx < \beta$, then $\tilde{w}$ has finite energy. Therefore $\tilde{w}$ is analytic in $D_1^+(0)$ and admits the following expansion in $D_{1}^+(0)$
\[
\tilde{w}(y)=\sum_{n=1}^\infty Re(\tilde{b}_n y^n), \ \tilde{b}_n \in \C, 
\quad \text{ so that } \quad
w(x)=\sum_{n=1}^\infty Re\left(\tilde{b}_n \frac{x^n}{|x|^{2n}}\right) \quad \text{ for } |x| > 1.
\]
By passing to polar coordinates and taking into account the symmetries of $w$, we find
\begin{equation}\label{eq:w_expansion}
w(r,\theta)=\sum_{n\text{ odd}} \frac{b_n}{r^n} \cos(n\theta), 
\quad r > 1, \quad b_n\in\R, \text{ with } b_1<0.
\end{equation}
Hence $w(r,\theta)=b_1\cos\theta/r+O(r^{-3})$ as $r\to\infty$, and an explicit calculation provides
\[
\lim_{r\to\infty} N(\tilde{\psi},r,0,0,A_e)=1 \quad \text{ and hence }\quad
\psi=\tilde{\psi}=e^{i\theta_e/2}(x_1+w).
\]
To conclude the proof of point (ii) it remains to show that $b_1=-\beta/\pi$. By testing the equation $- \Delta w = 0$ in $\R^2_+\setminus \Gamma_1$ by $w$ we deduce
\begin{equation}\label{eq:beta_boundary_integral1}
\beta=\int_{\R^2_+}|\nabla w|^2\dx = -2\int_{\Gamma_1} x_1\nabla w\cdot\nu \dsigma.
\end{equation}
On the other hand, by testing the equation for $w$ by $x_1$ in $D_R^+(0)$, $R>1$, and the equation $\Delta x_1=0$ by $w$ in $D_R^+(0)$ and subtracting we obtain
\begin{equation}\label{eq:beta_boundary_integral2}
\int_{\partial D_R^+(0)} (w\nabla x_1-x_1\nabla w)\cdot\nu \dsigma 
-2\int_{\Gamma_1} x_1\nabla w\cdot\nu \dsigma =0.
\end{equation}
By combining \eqref{eq:w_expansion}-\eqref{eq:beta_boundary_integral2} we obtain
\[
\begin{split}
\beta & =\lim_{R\to\infty} \int_{\partial D_R^+(0)} (x_1\nabla w-w\nabla x_1) \cdot\nu \dsigma \\
&=\lim_{R\to\infty} \left\{ -\sum_{n\text{ odd}} \frac{(n+1)b_n}{R^{n-1}} \int_{-\pi/2}^{\pi/2} \cos(n\theta)\cos\theta \,\mathrm{d}\theta \right\} \\
&=-2 b_1 \int_{-\pi/2}^{\pi/2} \cos^2\theta \,\mathrm{d}\theta
= -\pi b_1
\end{split}
\]
which concludes the proof.
\end{proof}

We can interpret $\beta$ as the cost, in terms of energy, needed to impose that $w$ vanishes on $\Gamma_1$, or equivalently as the energy cost of the nodal line of $\psi$.

\section{Pole approaching the boundary not on a nodal line of $\varphi_k$}
\label{sec:pole_approaching_not_nodal}

Let $\varphi^a$ be a solution of \eqref{eq:probleminhalfball}. In this section we treat the case when  $a\to0$ and $\varphi$ has a zero of order 1 at 0 (no nodal lines). In this case, if $\pi(a)=(0,a_2)$ is sufficiently close to $0$, then $\varphi$ has a zero of order 1 also at $\pi(a)$: there exists $\bar{a}_2>0$ such that, for $|\pi(a)|<\bar{a}_2$, we have
\begin{equation}\label{eq:phi_no_nodal_lines}
\varphi(|x-\pi(a)|,\theta_{\pi(a)})=|x-\pi(a)| \left(c_1(\pi(a)) \cos\theta_{\pi(a)} +d_1(\pi(a))\sin\theta_{\pi(a)} \right)+O(|x-\pi(a)|^2),
\end{equation}
as $|x-\pi(a)|\to0$, with $c_1(\pi(a))^2+d_1(\pi(a))^2\neq 0$ and $x - \pi(a) = |x-\pi(a)| (\cos \theta_{\pi(a)}, \sin \theta_{\pi(a)})$. In the following, we keep the notation used in Section \ref{sec:freq_formula} and we notice that the results proved therein hold for $\varphi^a$.

\subsection{Estimates on the frequency function} 
\begin{lemma} \label{lemma:boundNr0}
Let $\varphi^a$ be a solution of \eqref{eq:probleminhalfball} and suppose that $\varphi$ has a zero of order 1 at $0$. Let $|\pi(a)|=|a_2|<\bar{a}_2$ so that \eqref{eq:phi_no_nodal_lines} holds.
For every $\varepsilon > 0$, there exists $\tilde{r}_\eps>0$ such that for all $r_\eps\leq \tilde{r}_\eps$ there exists $\bar{a}_{1,\eps,r_\eps} > 0$ such that
\begin{align} \label{eq:boundNr0}
1\leq N(\varphi^a,r_{\eps},\pi(a),\lambda^a,A_a) \leq 1+\frac{\varepsilon}{2} \quad\text{ for all } a_1 < \bar{a}_{1,\eps,r_\eps}.
\end{align}
\end{lemma}
\begin{proof}
The bound from below can be proved as in \eqref{eq:Ngeq1}. Let us concentrate on the bound from above.
Relation \eqref{eq:phi_no_nodal_lines} implies that $1\leq N(\varphi,r,\pi(a),\lambda,0)\leq 1+O(r)$ as $r\to0$ (see for example \cite[Corollary 2.2.4]{HanLinBook}). Let $\tilde{r}_\eps$ be such that $N(\varphi,\tilde{r}_\eps,\pi(a),\lambda,0)\leq 1+\eps/8$. By the monotonicity property of the Almgren function for the eigenfunctions of the Laplacian (see for example \cite[Corollary 3.1.2]{HanLinBook}), we have that $N(\varphi,r,\pi(a),\lambda,0)\leq 1+\eps/4$ for every $r\leq \tilde{r}_\eps$. Fix $0<r_\eps<\tilde{r}_\eps$.
Since we know from \cite[Remark 4.4]{BonnaillieNorisNysTerracini2013} that
\[
\lambda^a\to\lambda \quad\text{and}\quad \|e^{-i \theta_a /2}\varphi^a- \varphi\|_{H^1(\Omega)} \to 0, \quad\text{as } a \to0,
\]
we deduce $|N(\varphi^a,r_{\eps},\pi(a),\lambda^a,A_a)-N(\varphi,r_\eps,\pi(a),\lambda,0)|\leq \eps/4$ for $|\pi(a)|<\bar{a}_2$ and $a_1<\bar{a}_{1,\eps,r_\eps}$.
\end{proof}

So far we have obtained an estimate on $N$ for a fixed radius $r_\eps$. Since $N$ is not increasing, this is not sufficient to obtain the estimate for $r\to0$. Nonetheless, we can provide a bound on $N$ for $r$ sufficiently far from the singularity. This is done by exploiting the estimates proved in the Section \ref{sec:freq_formula}.

\begin{lemma} \label{lemma:boundforbigradiusN}
Let $\varphi^a$ be a solution of \eqref{eq:probleminhalfball} and suppose that $\varphi$ has a zero of order 1 at $0$. Let $|\pi(a)|<\bar{a}_2$ so that \eqref{eq:phi_no_nodal_lines} holds.
For every $\varepsilon > 0$ there exist $r_\eps, \bar{a}_{1,\eps}>0$ and $k_\eps>1$ such that
\begin{align} \label{eq:boundN2}
N(\varphi^a,r,\pi(a),\lambda^a,A_a) \leq 1+\eps
\end{align}
for every $a_1<\bar{a}_{1,\eps}$ and for every $k_\eps a_1<r<r_\eps$, and
\begin{align} \label{eq:doublingformula}
\frac{H(\varphi^a,r_{2},\pi(a))}{H(\varphi^a,r_{1},\pi(a))} 
\leq \left( \frac{r_{2}}{r_{1}} \right)^{2(1+\eps)}.
\end{align}
for every $a_1<\bar{a}_{1,\eps}$ and $k_\eps a_1<r_1<r_2<r_\eps$.
\end{lemma}
\begin{proof}
To prove the first inequality we combine the previous lemma with Lemma \ref{lemma:boundN1}. In Lemma \ref{lemma:boundN1} we choose $r_0=r_\eps<\tilde{r}_\eps$. For every $k>1$, $a_1<\min\{r_\eps/k,\bar{a}_{1,\eps,r_\eps}\}$ and $ka_1<r<r_\eps$ we have
\[
N(\varphi^a,r,\pi(a),\lambda^a,A_a)\leq \left(2+\frac{\eps}{2}\right) \exp\left(\frac{Ce^{Cr_\eps^2}}{k^2}+\frac{2\lambda^ar_\eps^2}{1-\lambda^ar_\eps^2}\right)-1.
\]
We can impose that the right hand side above is less than $1+\eps$ by choosing $r_\eps$ sufficiently small and $k=k_\eps$ sufficiently large. Then we let $\bar{a}_{1,\eps}<\min\{r_\eps/k_\eps,\bar{a}_{1,\eps,r_\eps}\}$.

Let us look at the second inequality. We deduce from Lemma \ref{lemma:derivativeofH} and from \eqref{eq:boundN2} that 
\begin{align*}
\frac{\mathrm{d}}{\mathrm{d}r} \log H(\varphi^a,r,\pi(a)) = \frac{2}{r} N(\varphi^a,r,\pi(a),\lambda^a,A_a) \leq \frac{2 (1+\eps)}{r}
\end{align*}
for $a_1< \bar{a}_{1,\eps}$ and $k_\eps a_1<r<r_\eps$. Integrating between $r_1$ and $r_2$ we obtain the result.
\end{proof}
In Lemma \ref{lemma:inferiorboundH}, we obtained an superior bound on the function $H(\varphi^a,r,\pi(a))$. In the case where $r$ is sufficiently far from the singularity, we can obtain an inferior bound on $H(\varphi^a,r,\pi(a))$, which can be improved with respect to \eqref{eq:doublingformula}. This is the object of the following lemma.
\begin{lemma} \label{lemma:H_ka}
Let $\varphi^a$ be a solution of \eqref{eq:probleminhalfball} and suppose that $\varphi$ has a zero of order 1 at $0$. Let $|\pi(a)|<\bar{a}_2$ so that \eqref{eq:phi_no_nodal_lines} holds. 

For every $K>\sqrt{\beta/\pi}$ ($\beta$ defined in \eqref{eq:beta_def}) there exists $\bar{a}_1>0$ (depending on $K$) and $C>0$ such that
\[
H(\varphi^a,Ka_1,\pi(a)) \geq C (Ka_1)^2 \quad \text{for every } a_1<\bar{a}_1.
\]
\end{lemma}
\begin{proof}
We consider the function $\psi$ which has been introduced in the statement of Proposition \ref{prop:unique_limit_profile} (with abuse of notation we divide by the multiplicative constant $C$ which is not relevant in this context). The rescaled function
\[
\Phi(x)=a_1 \psi\left( \frac{x-\pi(a)}{a_1} \right)
\]
satisfies
\[
\left\{\begin{array}{ll}
(i\nabla+A_a)^2\Phi=0 \quad & \R^2_+ \\
\Phi=0 &\{x_1=0\},
\end{array}\right.
\]
and the following expansion, where $x-\pi(a)=\rho\, (\cos\theta_{\pi(a)}, \sin\theta_{\pi(a)})$:
\begin{equation}\label{eq:Phi_expansion}
\Phi(\rho,\theta_{\pi(a)})=e^{i\frac{\theta_a}{2}} \left( \rho \cos\theta_{\pi(a)}
-\frac{\beta}{\pi} a_1^2\frac{\cos\theta_{\pi(a)}}{\rho} 
+\sum_{\substack{n \geq 3 \\ n \text{ odd}}} b_n a_1^{n+1} \frac{\cos(n\theta_{\pi(a)})}{\rho^n} \right)
\quad \text{for } \rho>a_1.
\end{equation}
By testing the equation satisfied by $\varphi^a$ by $\Phi$ in $D_r^+(\pi(a))$ ($r>a_1$), we obtain:
\begin{equation}\label{eq:phia_tested_Phi}
\lambda^a \int_{D_r^+(\pi(a))} p(x) \varphi^a\overline{\Phi} \dx
=i \int_{\partial D_r^+(\pi(a))} \left\{ (i\nabla+A_a)\varphi^a\cdot\nu \, \overline{\Phi}
+\varphi^a \, \overline{(i\nabla+A_a)\Phi\cdot\nu} \right\}\dsigma.
\end{equation}
Fix $K>\sqrt{\beta/\pi}$. For $\rho>a_1$ we also define
\begin{equation}\label{eq:Gamma_def}
\Gamma(\rho,\theta_{\pi(a)})= e^{i\frac{\theta_a}{2}} \left\{
\left(K^2-\frac{\beta}{\pi}\right) a_1^2\frac{\cos\theta_{\pi(a)}}{\rho} 
+\sum_{\substack{n \geq 3 \\ n \text{ odd}}} b_na_1^{n+1}\frac{\cos(n\theta_{\pi(a)})}{\rho^n}\right\},
\end{equation}
so that
\[
\left\{\begin{array}{ll}
(i\nabla+A_a)^2\Gamma=0 \quad & \R^2_+ \setminus D_{a_1}^+(\pi(a))\\
\Gamma=0 &\{x_1=0\} \\
\Gamma=\Phi & \partial D_{Ka_1}^+(\pi(a)).
\end{array}\right.
\]
By testing the equation satisfied by $\varphi^a$ by $\Gamma$ in an annulus $(D_R^+\setminus D_r^+)(\pi(a))$ ($R>r>a_1$), we obtain:
\begin{equation}\label{eq:phia_tested_Gamma}
\lambda^a \int_{(D_R^+\setminus D_r^+)(\pi(a))} p(x) \varphi^a \overline{\Gamma} \dx 
=i \int_{\partial (D_R^+\setminus D_r^+)(\pi(a))} \left\{  (i\nabla +A_a)\varphi^a\cdot\nu \, \overline{\Gamma}+\varphi^a \, \overline{(i\nabla +A_a)\Gamma \cdot\nu}
\right\} \dsigma.
\end{equation}
In equations \eqref{eq:phia_tested_Phi} and \eqref{eq:phia_tested_Gamma} we choose $r=Ka_1$ and $R>Ka_1$ to be fixed later. Adding the two equations we obtain
\begin{equation}\label{eq:PhiGamma}
\begin{split}
i \int_{\partial D_{Ka_1}^+(\pi(a))} \varphi^a \left\{ \overline{(i\nabla+A_a)\Phi\cdot\nu}
-\overline{(i\nabla +A_a)\Gamma\cdot\nu} \right\} \dsigma
=\lambda^a \int_{D_{Ka_1}^+(\pi(a))}  p(x) \varphi^a\overline{\Phi} \dx \\
+\lambda^a \int_{(D_R^+\setminus D_{Ka_1}^+)(\pi(a))} p(x) \varphi^a\overline{\Gamma} \dx
-i \int_{\partial D_R^+(\pi(a))} \left\{ (i\nabla+A_a)\varphi^a\cdot\nu\overline{\Gamma}
+\varphi^a\overline{(i\nabla+A_a)\Gamma\cdot\nu} \right\} \dsigma.
\end{split}
\end{equation}
Noticing that
\[
(i\nabla+A_a)\Phi\cdot\nu\mid_{\partial D^+_{Ka_1(\pi(a))}} 
=i e^{i\frac{\theta_a}{2}} \Big\{ \Big(1+ \frac{\beta}{\pi K^2}\Big)\cos\theta_{\pi(a)}
-\sum_{\substack{n \geq 3 \\ n \text{ odd}}} \frac{n b_n}{K^{n+1}} \cos(n\theta_{\pi(a)}) \Big\},
\]
\[
(i\nabla+A_a)\Gamma\cdot\nu\mid_{\partial D^+_{Ka_1(\pi(a))}} 
=- i e^{i\frac{\theta_a}{2}} \Big\{ \Big(1- \frac{\beta}{\pi K^2}\Big)\cos\theta_{\pi(a)}
+\sum_{\substack{n \geq 3 \\ n \text{ odd}}} \frac{n b_n}{K^{n+1}} \cos(n\theta_{\pi(a)}) \Big\},
\]
we can estimate the left hand side of \eqref{eq:PhiGamma} from above as follows
\begin{equation}\label{eq:PhiGamma_lhs}
\begin{split}
\left| i \int_{\partial D_{Ka_1}^+(\pi(a))} \varphi^a \left\{ \overline{(i\nabla+A_a)\Phi\cdot\nu}
-\overline{(i\nabla +A_a)\Gamma\cdot\nu} \right\} \dsigma \right| 
=\left| \int_{\partial D^+_{Ka_1}(\pi(a))} \varphi^a e^{-i\frac{\theta_a}{2}} 
2 \cos\theta_{\pi(a)}\dsigma \right| \\
\leq 2\|\varphi^a\|_{L^2(\partial D^+_{Ka_1}(\pi(a)))} 
\|\cos\theta_{\pi(a)}\|_{L^2(\partial D^+_{Ka_1}(\pi(a)))}
=K a_1\sqrt{ 2\pi H(\varphi^a,Ka_1,\pi(a)) }.
\end{split}
\end{equation}
In what follows we will estimate the right hand side of \eqref{eq:PhiGamma}. To this aim, recall that for every $r>0$ it holds
\[
\| e^{-i\frac{\theta_a}{2}} \varphi^a-\varphi\|_{C^\infty(\Omega\setminus D_r^+(0))} \to0
\quad \text{as } a\to0.
\]
Moreover, $\varphi$ satisfies \eqref{eq:phi_no_nodal_lines}. Hence we have
\begin{equation}\label{eq:phi_rho}
\varphi^a\mid_{\partial D^+_\rho(\pi(a))} 
=e^{i\frac{\theta_a}{2}} {c\rho\cos\theta_{\pi(a)}+h(\rho,\theta_{\pi(a)})} +o_{a_1}(1),
\quad \text{ for every } \rho>a_1,
\end{equation}
where $c\in\R$ and $h$ satisfies (see \eqref{eq:g_remainder})
\begin{equation}\label{eq:h_remainder}
\lim_{\rho\to0} \frac{\|h(\rho,\cdot)\|_{C^1(\partial D^+_\rho(\pi(a)))}}{\rho}=0.
\end{equation}
Let's first look at the boundary term in the right hand side of \eqref{eq:PhiGamma}. Taking into account that $R$ is fixed and $a_1\to0$, we have
\[
(i\nabla+A_a)\varphi^a\cdot\nu \, \overline{\Gamma}\mid_{\partial D^+_R(\pi(a))}
=i \left( K^2-\frac{\beta}{\pi} \right) \frac{a_1^2}{R} \left\{ c\cos^2\theta_{\pi(a)} +\frac{\partial h}{\partial\rho}(R,\theta_{\pi(a)}) \cos\theta_{\pi(a)} \right\} +o(a_1^2),
\]
\[
\varphi^a \, \overline{(i\nabla+A_a)\Gamma\cdot\nu}\mid_{\partial D^+_R(\pi(a))}
=i \left( K^2-\frac{\beta}{\pi} \right) \frac{a_1^2}{R} \left\{ c\cos^2\theta_{\pi(a)} 
+\frac{h(R,\theta_{\pi(a)})}{R} \cos\theta_{\pi(a)} \right\} +o(a_1^2),
\]
so that
\begin{equation}\label{eq:PhiGamma_rhs_boundary}
\begin{split}
-i \int_{\partial D_R^+(\pi(a))} \left\{ (i\nabla+A_a)\varphi^a\cdot\nu\overline{\Gamma}
+\varphi^a\overline{(i\nabla+A_a)\Gamma\cdot\nu} \right\} \dsigma \\
=c (\pi K^2-\beta) a_1^2 
+ \left( K^2-\frac{\beta}{\pi} \right) \frac{a_1^2}{R} \int_{\partial D_R^+(\pi(a))} 
\left( \frac{h(R,\theta_{\pi(a)} )}{R}+\frac{\partial h}{\partial\rho}(R,\theta_{\pi(a)} ) \right) \cos\theta_{\pi(a)}  \dsigma +o(a_1^2) \\
\geq c (\pi K^2-\beta) a_1^2 -C a_1^2 \left\| \frac{h(R,\cdot )}{R}+\frac{\partial h}{\partial\rho}(R,\cdot )  \right\|_{L^\infty(\partial D^+_R(\pi(a)))} +o(a_1^2) 
\geq C' K^2 a_1^2,
\end{split}
\end{equation}
for suitable $C',R>0$ and $a_1$ sufficiently small, thanks to \eqref{eq:h_remainder}.

Concerning the integral in the annulus in \eqref{eq:PhiGamma}, we replace \eqref{eq:Gamma_def} and \eqref{eq:phi_rho} to obtain
\begin{equation}\label{eq:PhiGamma_rhs_annulus}
\begin{split}
\left|\int_{(D_R^+\setminus D_{Ka_1}^+)(\pi(a))} p(x) \varphi^a\overline{\Gamma} \dx \right| 
\leq \|p\|_{L^\infty} \left| c\frac{\pi}{4}\left(K^2-\frac{\beta}{\pi}\right) a_1^2R^2 \right| \\
+ \|p\|_{L^\infty} \|h\|_{L^\infty} \int_{Ka_1}^R\int_{\partial D^+_\rho(\pi(a))} \left|
\sum_{\substack{n \geq 3 \\ n \text{ odd}}} b_n a_1^{n+1}\frac{\cos(n\theta_{\pi(a)})}{\rho^n} \right| \dsigma \drho +o(a_1^2)  \\
\leq C \left\{ a_1^2 K^2 R^2 + \sum_{\substack{n \geq 3 \\ n \text{ odd}}} |b_n| a_1^{n+1}\left| \frac{1}{R^{n-2}}-\frac{1}{(Ka_1)^{n-2}} \right|  \right\} +o(a_1^2)
\leq C a_1^2 K^2 R^2 +o(a_1^2),
\end{split}
\end{equation}
since $R,K$ are fixed while $a_1\to0$.

In order to estimate the last term, we apply Lemma \ref{lemma:Poincaré}, the equation satisfied by $\Phi$ and the expansion \eqref{eq:Phi_expansion} with $\rho=Ka_1>a_1$, as follows
\[
\begin{split}
\|\Phi\|^2_{L^2(D^+_{Ka_1}(\pi(a)))} 
\leq Ka_1 \int_{\partial D^+_{Ka_1}(\pi(a))} |\Phi|^2 \dsigma 
+(Ka_1)^2 \int_{D^+_{Ka_1}(\pi(a))} |(i\nabla+A_a)\Phi|^2 \dx \\
= Ka_1 \int_{\partial D^+_{Ka_1}(\pi(a))} |\Phi|^2 \dsigma 
-i (Ka_1)^2 \int_{\partial D^+_{Ka_1}(\pi(a))} (i\nabla+A_a)\Phi\cdot\nu \overline{\Phi}\dsigma 
=O(a_1^4).
\end{split}
\]
In a similar way
\[
\begin{split}
\|\varphi^a\|^2_{L^2(D^+_{Ka_1}(\pi(a)))} 
\leq Ka_1 \int_{\partial D^+_{Ka_1}(\pi(a))} |\varphi^a|^2 \dsigma \\
+(Ka_1)^2\left\{\lambda^a \int_{D^+_{Ka_1}(\pi(a))} p(x) |\varphi^a|^2\dx 
-i \int_{\partial D^+_{Ka_1}(\pi(a))} (i\nabla+A_a)\varphi^a\cdot\nu \overline{\varphi^a} \dsigma \right\},
\end{split}
\]
so that, using \eqref{eq:phi_rho},
\[
(1-\lambda^a \|p\|_{L^\infty} (Ka_1)^2) \|\varphi^a\|^2_{L^2(D^+_{Ka_1}(\pi(a)))} =O(a_1^4).
\]
The H\"older inequality provides
\begin{equation}\label{eq:PhiGamma_rhs_ball}
\left| \int_{D^+_{Ka_1}(\pi(a))} \varphi^a \overline{\Phi} \dx \right| \leq 
\|\varphi^a\|_{L^2(D^+_{Ka_1}(\pi(a)))} \|\Phi\|_{L^2(D^+_{Ka_1}(\pi(a)))} =O(a_1^{4}).
\end{equation}
By combining \eqref{eq:PhiGamma}, \eqref{eq:PhiGamma_lhs}, \eqref{eq:PhiGamma_rhs_boundary}, \eqref{eq:PhiGamma_rhs_annulus} and \eqref{eq:PhiGamma_rhs_ball}, we obtain
\[
Ka_1 \sqrt{2\pi H(\varphi^a,Ka_1,\pi(a))} \geq C(Ka_1)^2 -C'(K a_1 R)^2+o(a_1^2) 
\geq C'' (Ka_1)^2,
\]
for a suitable choice of $R$ and for $a_1$ sufficiently small, and hence the thesis.
\end{proof}
Lemma \ref{lemma:inferiorboundH} and \ref{lemma:H_ka} allow us to say that $H(\varphi^a, Ka_1, \pi(a)) = C (Ka_1)^2$ for $K > \max\{ \beta/\pi, 1\}$ and $a_1 < \bar{a}_1$ ($\bar{a}_1$ defined in Lemma \ref{lemma:H_ka}).

\subsection{Normalized blow-up at the pole}
In order to analyze the behavior of $\varphi^a$ near $a$ (for $|a|$ close to 0), we perform a normalized blow-up of the function near the pole. For a fixed $\eps>0$, let
\begin{equation}\label{eq:r_eps_def}
r_\eps, \bar{a}_{1,\eps},  k_\eps \text{ be as in Lemma \ref{lemma:boundforbigradiusN}}.
\end{equation}
We define
\begin{align} \label{eq:blowup}
\psi^a(y) = \frac{\varphi^a(a_1 y + \pi(a))}{\sqrt{H(\varphi^a,k_\eps a_1,\pi(a))}}, \quad y\in D_{R_0}^+(0), \ R_0=\frac{r_0}{a_1}.
\end{align}

Note that these are the functions which appear in the statement of Theorem \ref{theorem:blow_up_almgren} (with $K=k_\eps$) and that they are singular at $e=(0,1)$, independently of $a$.
We also remark that $\psi^a$ solves the problem
\begin{align} \label{eq:problemblowup}
\left\{ \begin{aligned}  (i \nabla + A_e)^2 \psi^a & = \lambda^a a_1^{2} \hat{p}(y) \psi^a \quad & D_{R_0}^+(0) \\
                         \psi^a & = 0  \quad & \{y_1=0\},
                        \end{aligned} \right.
\end{align}
where $\hat{p}(y) = p(a_1 y + \pi(a))$.

A direct calculation provides the following relations between the frequency function for $\varphi^a$ and that for $\psi^a$
\begin{align} 
& E(\psi^a,R,0,\lambda^a a_1^2,A_e) = \frac{E(\varphi^a,Ra_1,\pi(a),\lambda^a,A_a)}{H(\varphi^a,k_\eps a_1, \pi(a))},
\label{eq:equalityEblowup} \\
& H( \psi^a,R, 0) = \frac{H(\varphi^a,Ra_1, \pi(a))}{H(\varphi^a,k_\eps a_1, \pi(a))}, 
\label{eq:equalityHblowup} \\
& N(\psi^a,R,0,\lambda^a a_1^2, A_e) = N(\varphi^a,Ra_1,\pi(a),\lambda^a,A_a),  
\label{eq:equalityN}
\end{align}
for $R>1$. Here, with an abuse of notation, the frequency function for $\varphi^a$ contains the weight $p(x)$, while in the frequency function for $\psi^a$ appears $\hat{p}(y)$ due to the change of variables in the integral. This has no influence in the calculations, since both $p$ and $\hat{p}$ satisfy \eqref{eq:weight_assumptions}.

We will show that the boundedness of the Almgren's function implies the convergence of the blow-up sequence as $a_1\to0$.
To this aim, notice that Lemma \ref{lemma:boundforbigradiusN} and relations \eqref{eq:equalityEblowup}-\eqref{eq:equalityN} provide the following bounds.

\begin{lemma} \label{lemma:Nboundblowup2}
Given $\eps >0$, take the same assumptions and notations of Lemma \ref{lemma:boundforbigradiusN}. Let $\psi^a$ be as in \eqref{eq:blowup}.
Then
\begin{align} \label{eq:Nboundblowup2}
N( \psi^a, R,0,\lambda^a a_1^2, A_e) \leq 1 + \eps
\end{align}
for every $a_1 < \bar{a}_{1,\eps}$ and $k_\eps < R < r_\eps / a_1$, and
\begin{align} \label{eq:doublingformulablowup}
\frac{H(\psi^a,R_{2},0)}{H(\psi^a,R_{1},0)} 
\leq \left( \frac{R_{2}}{R_{1}} \right)^{2(1+\eps)}.
\end{align}
for every $a_1<\bar{a}_{1,\eps}$ and $k_\eps <R_1<R_2<r_\eps / a_1$.
\end{lemma}

\begin{lemma} \label{lemma:normbound}
Given $\eps >0$, take the same assumptions and notations of Lemma \ref{lemma:boundforbigradiusN}. Let $\psi^a$ be as in \eqref{eq:blowup}.
For every $R > k_\eps$, there exists a constant $C(\eps,R) > 0$ such that
\begin{align}
\| \psi^a \|_{H^{1}_{A_e}(D_{R}^{+}(0))} \leq C(\eps,R) \qquad 
\text{for every } a_1<\min\left\{\frac{r_\eps}{R},\bar{a}_{1,\eps}\right\}  .
\end{align}
\end{lemma}
\begin{proof}
Relation \eqref{eq:doublingformulablowup} and our choice of the normalization provide
\begin{align}\label{eq:H_psi_bound}
H(\psi^a, R,0) = \frac{H(\psi^a,R,0)}{H(\psi^a,k_\eps,0)} \leq \left(\frac{R}{k_\eps}\right)^{2 (1 + \eps)}\leq C(\eps) R^{2 (1 + \eps)}.
\end{align}
This, together with the definition of $N$ and \eqref{eq:Nboundblowup2}, implies
\[
E(\psi^a,R,0,\lambda^a a_1^2, A_e) = N(\psi^a,R,0,\lambda^a a_1^2, A_e) H(\psi^a,R,0)
\leq  C(\eps) R^{2 (1 + \eps)}.
\]
Both relations hold for $R>k_\eps$ and $a_1<\min\{ r_\eps/R, \bar{a}_{1,\eps} \}$.
Then
\begin{align*}
\int_{D_{R}^{+}(0)} | (i \nabla + A_e) \psi^a |^{2}\dy  
\leq C(\eps) R^{2 (1 + \eps)} + \lambda^a a_1^{2} \int_{D_{R}^{+}(0)} \hat{p}(y) |\psi^a|^{2}\dy \\
             \leq C(\eps) R^{2 (1 + \eps)} + \lambda^a a_1^{2} \|p\|_{L^\infty} R^2 \left( H( \psi^a,R,0) + \int_{D_{R}^{+}(0)} |(i \nabla + A_e) \psi^a|^{2} \dy \right) \\
             \leq C(\eps) R^{2 (1 + \eps)} + \lambda^a \|p\|_{L^\infty} r_\eps^2 C(\eps) R^{2 (1 + \eps)} + \lambda^a \|p\|_{L^\infty} r_\eps^{2} \int_{D_{R}^{+}(0)} |(i \nabla + A_e) \psi^a|^{2}\dy.
\end{align*}
At the second line we used the Poincaré inequality \eqref{eq:Poincaré}, at the third line we used \eqref{eq:H_psi_bound} and the fact that $R \leq r_{\eps}/a_1$. Then, thanks to \eqref{eq:r_0}, we have
\begin{align*}
\int_{D_{R}^{+}(0)} |(i \nabla + A_e) \psi^a|^{2} \, \dy \leq C(\eps) R^{2 (1 + \eps)} \frac{ 1 + \lambda^a \|p\|_{L^\infty} r_{\eps}^{2}}{1 - \lambda^a \|p\|_{L^\infty} r_{\eps}^{2}}.
\end{align*}
We look then at the second part of the norm. Using Poincaré inequality \eqref{eq:Poincaré}, we obtain
\begin{align*}
\int_{D_R^+(0)} |\psi^a|^2\dy  \leq R^2 H(\psi^a, R,0) + R^2 \int_{D_R^+(0)} | (i \nabla +A_e) \psi^a |^2\dy  \leq C(\eps,R),
\end{align*}
where we used the previous inequality and \eqref{eq:H_psi_bound}.
Finally, we combine the two contributions and obtain a constant depending only on $R$ and $\eps$.
\end{proof}

\begin{lemma} \label{lemma:weakconvergence}
Given $\eps >0$, take the same assumptions and notations of Lemma \ref{lemma:boundforbigradiusN}. Let $\psi^a$ be as in \eqref{eq:blowup}.
There exists $\psi \in H^1_{A_e,loc}(\mathbb{R}^{2}_+)$, $\psi\not\equiv0$, such that for every $R > k_\eps$ we have, up to a subsequence, $\psi^a \to \psi$ in $H^{1}_{A_e}(D_{R}^{+}(0))$ as $|a|\to0$. Moreover, $\psi$ solves
\begin{align} \label{eq:limitmagneticproblem}
\left\{ \begin{aligned} (i \nabla + A_e)^2 \psi& = 0 \quad & \mathbb{R}^{2}_+ \\
                            \psi & = 0 \quad & \{y_1=0\}.
                            \end{aligned} \right. 
\end{align}
\end{lemma}
\begin{proof}
By Lemma \ref{lemma:normbound}, there exists $\psi$ such that, up to a subsequence, $\psi^a \rightharpoonup \psi$ in $H^1_{A_e}(D_R^+(0))$ and $\psi^a \rightarrow \psi$ in $L^2(D_R^+(0))$ as $|a| \to 0$.
Due to the compactness of the trace embedding, we have $\int_{\partial D_{k_\eps}^+(0)} |\psi|^2\dsigma=k_\eps$, so that $\psi\not\equiv0$.
For every $R > k_\eps$ and for every test function $\phi \in C_0^\infty(D_R^+(0) \backslash \{e\})$, we have
\begin{align*}
\int_{D_R^+(0)} (i \nabla + A_e) \psi^a \cdot \overline{(i \nabla + A_e) \phi}\dy = \lambda^a a_1^2 \int_{D_R^+(0)} \hat{p}(y) \psi^a \bar{\phi} \dy.
\end{align*}
By the weak convergence in $H^1_{A_e}(D_R^+(0))$, the first term converges
\begin{align*}
\int_{D_R^+(0)} (i \nabla + A_e) \psi^a \cdot \overline{(i \nabla + A_e) \phi}\dy \rightarrow \int_{D_R^+(0)} (i \nabla + A_e) \psi \cdot \overline{(i\nabla + A_e)\phi} \dy.
\end{align*}
We estimate the second term as follows by means of Lemma \ref{lemma:normbound}
\begin{align*}
\left| \lambda^a a_1^2 \int_{D_R^+(0)} \hat{p}(y) \psi^a \bar{\phi} \dy\right|
&\leq \lambda^a a_1^2 \|p\|_{L^\infty} \|\phi\|_{L^2(D_R^+(0))} \|\psi^a\|_{L^2(D_R^+(0))} \\
&\leq C a_1^2 \| \psi^a \|_{H^1_{A_e}(D_R^+(0))} \leq C(\eps,R) a_1^2 \rightarrow 0,
\end{align*}
so that $\psi$ solves the limit equation \eqref{eq:limitmagneticproblem}. In order to prove the strong convergence, we consider the equation satisfied by $\psi^a-\psi$. We have
\[
(i \nabla + A_e)^2 (\psi^a - \psi) = \lambda^a a_1^{2} \hat{p}(y) \psi^a  \quad 
\text{in } D_{R}^{+}(0).
\]
By Lemma \ref{lemma:normbound} and the Sobolev embeddings, the right hand side above converges to 0 in $L^p(D_R^+(0))$ for every $p<\infty$ as $|a|\to0$. The Kato inequality
\[
-\Delta|\psi^a-\psi| \leq |(i\nabla+A_e)^2(\psi^a-\psi)| 
\]
(see for example \cite{Kato1972}) and the standard regularity theory for elliptic equations, imply that $|\psi^a-\psi|\to0$ in $W^{2,p}(D_R^+(0))$ for every $p<\infty$ as $|a|\to0$. This in turn implies that the convergence is $C^{1,\tau}_{loc}(D_R^+(0)\setminus\{e\})$ for every $\tau\in (0,1)$ and $H^1(D_r^+(0))$.
\end{proof}

As a consequence of the strong convergence and of Lemma \ref{lemma:Nboundblowup2}, we deduce the following.

\begin{lemma} \label{lemma:boundNpsi}
Let $\psi$ be defined in Lemma \ref{lemma:weakconvergence}. We have
\begin{align} \label{eq:boundNpsi}
 N(\psi, R,0, 0, A_e)  \leq 1 + \eps \quad\text{ for every } R>k_\eps,
\end{align}
\begin{align} \label{eq:doublingpsi}
\frac{H(\psi, R_2,0)}{H(\psi, R_1,0)} \leq \left( \frac{R_2}{R_1} \right)^{2(1+\eps)} \quad\text{ for every } k_\eps < R_1 < R_2.
\end{align}
\end{lemma}

\begin{lemma}\label{lemma:d_limit}
Let $\psi$ be defined in Lemma \ref{lemma:weakconvergence}. There exists $d \, \in [0, + \infty]$ such that
\[
\lim_{R \to + \infty} N(\psi, R,0, 0, A_e)=d.
\]
\end{lemma}
\begin{proof}
Reasoning as in \eqref{eq:N_derivative} we find
\begin{align*}
\frac{\mathrm{d}}{\mathrm{d}R} N(\psi, R, 0,0, A_e) \geq -\frac{2|M|}{RH(\psi,R,0)},
\end{align*}
where $M$ is now the constant
\[
M=\lim_{\eps\to0}Re\int_{\partial D_\eps(e)} (i\nabla+A_e)\psi\cdot\nu \overline{(i\nabla+A_e)\psi\cdot y} \dsigma.
\]
We can prove as in Lemma \ref{lemma:inferiorboundH} that
\[
\frac{H(\psi,R,0)}{H(\psi,k_\eps,0)}\geq \left(\frac{R}{k_\eps}\right)^2,
\]
for $R>k_\eps$.
Recalling that $H(\psi,k_\eps,0)=1$, we obtain
\begin{align*}
\frac{\mathrm{d}}{\mathrm{d}R} N(\psi, R, 0,0, A_e) \geq - \frac{C k_\eps^2}{R^3},
\end{align*}
for a positive constant $C$. Let us show that this implies the existence of the limit. Let for the moment $N(R)=N(\psi, R, 0,0, A_e)$. Integrating the last inequality in $(R_1,R_2)$, with $k_\eps<R_1<R_2$, we obtain
\begin{equation}\label{eq:NR_2minusNR_1}
N(R_2)-N(R_1) \geq C k_\varepsilon^2 \left(\frac{1}{R_2^2}-\frac{1}{R_1^2}\right).
\end{equation}
If $d=+\infty$ there is nothing to prove. Otherwise, we claim that $N(R)$ is bounded. Indeed, $d\neq\infty$ implies the existence of $K>0$ and of a sequence $R_n\to\infty$ such that $N(R_n)<K$ for every $n$, so that for $R$ sufficiently large and $R_n>R$ we have, by \eqref{eq:NR_2minusNR_1}
\[
N(R)\leq N(R_n)-C k_\varepsilon^2 \left(\frac{1}{R_n^2}-\frac{1}{R^2}\right)
\leq K+o(1) \quad \text{as } R\to\infty,
\]
so that $N$ is bounded. Suppose by contradiction that $N(R)$ does not admit limit $d\in [0,\infty)$. Then for every $\delta>0$ there exists a sequence $R_n\to\infty$ such that $|N(R_n)-N(R_{n+1})|\geq\delta$. The case $N(R_n)\geq N(R_{n+1})+\delta$ contradicts \eqref{eq:NR_2minusNR_1} if $R_n$ is great enough, the case $N(R_{n+1}) \geq N(R_n)+\delta$ contradicts the fact that $N$ is bounded.
\end{proof}

In the next subsection we will prove that $d=1$.

\subsection{Proof of Theorem \ref{theorem:blow_up_almgren}}

In order to study the behavior of the limit function $\psi$ at infinity, we perform a rescaling (blow-down) on the independent variable by a factor $R$. As before, using the boundedness of the Almgren's frequence of $\psi$, we prove the convergence of the blow-down function as $R \to \infty$. Moreover, we will prove that the limit function is an homogeneous harmonic function of degree $1$. Then, this aims us to conclude that sufficiently far from the singularity $\psi$ behaves like an harmonic function, up to a complex phase. More specifically, we prove that this function $\psi$ verifies the conditions of Proposition \ref{prop:unique_limit_profile}.  

\begin{lemma} \label{lemma:normboundblowdown}
Let $\psi$ be the function introduced in Lemma \ref{lemma:weakconvergence}. We define
\begin{align} \label{eq:blowdown}
w_R(x) = \frac{\psi(Rx)}{\sqrt{H(\psi,R,0)}}.
\end{align}
For every $r>0$ there exists a constant $C(\eps,r)$ such that $\| w_R \|_{H^1_{A_{e/R}}(D_r^+(0))} \leq C(\eps,r)$ for every $R>k_\eps$.
\end{lemma}
\begin{proof}
For $r>1$ and $R > k_\eps$ we have
\[
N(w_R, r,0, 0, A_{e/R}) = N(\psi, rR,0, 0, A_e)\leq 1 + \eps
\quad\text{and}\quad
H(w_R, r,0) = \frac{H(\psi,rR,0)}{H(\psi,R,0)} \leq r^{2 (1 + \eps)},
\]
by Lemma \ref{lemma:boundNpsi}. By combining the two, we obtain
\[
E(w_R, r, 0,0, A_{e/R}) \leq (1 + \eps) r^{2 ( 1 + \eps)}
\]
for every $r>1$ and $R > k_\eps$. As a consequence, using Lemma \ref{lemma:Poincaré} we estimate
\[
\| w_R \|_{H^1_{A_{e/R}}(D_r^+(0))}\leq (1+r^2)E(w_R,r,0,0,A_{e/R})+ r^2H(w_R,r,0) \leq C(\eps,r)
\]
for $R > k_\eps$.
\end{proof}

\begin{lemma}
Let $w_R$ be defined in \eqref{eq:blowdown}. There exists $w \in H^1_{loc}(\R^2_+)$, $w\not\equiv  0$, such that $e^{-i\theta_{e/R}/2} w_R \rightharpoonup w$ in $H^1_{loc}(\R^2_+)$. 
In addition, $w$ is harmonic in $\R^2_+$ with zero boundary conditions and, for almost every $r>0$, we have
\begin{equation}\label{eq:blow_down_limit_E}
\lim_{R\to\infty} E(w_R,r,0,0,A_{e/R}) = E(w,r,0,0,0).
\end{equation}
\end{lemma}
\begin{proof}
Fix $r>0$. By \eqref{eq:H_1_A_a_characterization} and Lemma \ref{lemma:normboundblowdown}, there exists a constant $C(\eps,r)>0$ (not depending on $R$) such that
\begin{align*}
\| w_R \|_{H^1(D_{r}^+(0))} \leq C \| w_R \|_{H^1_{A_{e/R}}(D_r^+(0))} \leq C(\eps, r).
\end{align*}
In order to check that the constant $C$ in the previous inequality does not depend on the position of the singularity $e/R$, one can extend functions in $H^1_{A_{e/R}}(D_r^+(0))$ which vanish on $\{x_1=0\}$ trivially to functions belonging to $H^1_{A_{e/R}}(D_r(0))$, and then proceed as in the proof of \cite[Lemma 7.4]{MOR}.
Hence there exists $\tilde{w} \in H^1(D_r^+(0))$ such that $w_R \rightharpoonup \tilde{w} $ in $H^1(D_r^+(0))$ and $w_R \rightarrow \tilde{w} $ in $L^2(D_r^+(0))$, as $R \to + \infty$. Since $H(w_R,1,0)=1$ for every $R$, the trace embeddings provide $\tilde{w}\not\equiv0$. 

Let $w=e^{-i\theta_{0}/2}\tilde{w}$. In order to prove that $w$ is harmonic, notice first that $(i \nabla + A_{e/R})^2 w_R = 0$ in $\R^2_+$ for every $R$. Given a test function $\phi\in C_0^{\infty}(D_r^+(0))$, let $R$ be so large that $e/R\not\in \text{supp}\{\phi\}$. Consequently we have
\begin{equation}\label{eq:blow_up_harmonic_away0}
-\Delta(e^{-i\theta_{e/R}/2} w_R)=0
\end{equation}
in $\text{supp}\{\phi\}$. This implies, using the weak convergence,
\[
0 = \int_{D_r^+(0)}\nabla (e^{-i\theta_{e/R}/2} w_R) \cdot \nabla \phi \dx \rightarrow \int_{D_r^+(0)} \nabla w \cdot \nabla \phi \dx \quad\text{ as } R\to\infty,
\]
so that $w$ is harmonic in $D_r^+(0)$.

To prove the last part of the statement, fix two concentric semi-annuli $\mathcal{A}_1\subset \mathcal{A}_2$, centered at the origin and having positive distance from it. Let $\eta$ be a cut-off function which is 1 in $\mathcal{A}_1$ and vanishes outside $\mathcal{A}_2$. For $R$ sufficiently large, we have that \eqref{eq:blow_up_harmonic_away0} holds in $\mathcal{A}_2$. By testing the equation satisfied by $e^{-i\theta_{e/R}/2} w_R-w$ by $(e^{-i\theta_{e/R}/2} w_R-w)\eta$ in $\mathcal{A}_2$, we obtain
\[
\begin{split}
\int_{\mathcal{A}_1} |\nabla(e^{-i\theta_{e/R}/2} w_R-w)|^2 \dx
\leq \left|\int_{\mathcal{A}_2} \nabla(e^{-i\theta_{e/R}/2} w_R-w)\nabla\eta(e^{-i\theta_{e/R}/2} w_R-w) \dx  \right| \to 0,
\end{split}
\]
which tends to 0 as $R\to\infty$ by the weak convergence. This implies that
\begin{equation}\label{eq:convergence_boundary_D_rho}
\int_{\partial D_{\rho}^+(0)}\left( | \nabla (e^{-i\theta_{e/R}/2} w_R - w)|^2 + |e^{-i\theta_{e/R}/2} w_R - w |^2 \right)\dsigma \rightarrow 0,
\end{equation}
for almost every $\rho$ such that $\partial D_\rho^+(0)\subset\mathcal{A}_1$, as $R \to + \infty$. 

Finally, we use integration by parts as follows (the second equality is well defined provided for $R>1/r$)
\[
\begin{split}
|E(w_R,r,0,0,A_{e/R}) - E(w,r,0,0,0)|
\leq \int_{\partial D_r^+(0)} \left| -i(i\nabla+A_{e/R})w_R\cdot\nu\bar{w}_R - \nabla w\cdot\nu w \right|\dsigma \\
= \int_{\partial D_r^+(0)} \left| \nabla(e^{-i\theta_{e/R}/2} w_R)\cdot\nu \overline{e^{-i\theta_{e/R}/2} w_R} -\nabla w\cdot\nu w  \right|\dsigma \to 0,
\end{split}
\]
where the convergence to 0 comes from \eqref{eq:convergence_boundary_D_rho} for almost every $r > 0$.
\end{proof}

\begin{proof}[End of the proof of theorem \ref{theorem:blow_up_almgren}]
By combining \eqref{eq:blow_down_limit_E} and Lemma \ref{lemma:d_limit} we obtain, for almost every $r>0$,
\[
N(w,r,0,0,0)=\lim_{R\to\infty} N(w_R,r,0,0,A_{e/R}) = \lim_{R\to\infty} N(\psi,rR,0,0,A_e)=d
\]
(recall that $\psi$ was introduced in Lemma \ref{lemma:weakconvergence}). Since $N(w,\cdot,0,0,0)$ is continuous, it is constant. Since we proved in the previous lemma that $w$ is harmonic with zero boundary conditions on $\{x_1=0\}$, we deduce from standard arguments (see for example \cite[Proposition 3.9]{NTTV2010}) that $w(r,\theta)=Cr^d\cos(d\theta)$, for some $d\in\N$ odd. Comparing with \eqref{eq:boundNpsi}, taking for example $\eps=1/2$, we conclude that $d=1$. In conclusion, by Proposition \ref{prop:unique_limit_profile}, $\psi$ solves \eqref{eq:limit_profile}-\eqref{eq:limit_profile_normalization}. Then, $\psi = C e^{i \theta_e/2} \left( r \cos \theta - \frac{\beta}{\pi} \frac{\cos \theta}{r} + O(r^{-3}) \right)$ for $r > 1$. Moreover, since $H(\psi, k_\eps, 0) = 1$, the constant $C$ is given by
\begin{align} \label{eq:constante_normalisation}
C^2 = \frac{1}{\frac{k_\eps^2 \pi}{2} - \beta + O(\frac{1}{k_\eps^2})}.
\end{align}

\end{proof}

\subsection{Proof of Theorem \ref{theorem:no_nodal_lines}}
\label{subsec:theorem_no_nodal_lines}
We can assume without loss of generality that $b=0$ and moreover, by Lemma \ref{lemma:Riemannmapping}, that $\Omega$ satisfies \eqref{eq:rectified_domain}.
Let $\varphi_k$ have a zero of order $1$ at $0\in\partial\Omega$, meaning that there are no nodal lines of $\varphi_k$ ending at $0$. 
Let $K > \sqrt{\beta/\pi}$ large be such that the statement of Theorem \ref{theorem:blow_up_almgren} holds.
We proceed similarly to the proof of Theorem \ref{theorem:at_least_one_nodal_line}.

For $i=1,\ldots,k$ let
\[
v_{i}^{ext} = e^{-i \frac{\theta_{a}}{2}} \varphi_{i}^{a} \quad \text{ in } \Omega\setminus D_{Ka_1}(\pi(a)).
\]
For $a_1$ sufficiently small, $v_i^{int}$ is defined as the unique function which achieves
\[
\inf\left\{ \int_{D_{Ka_1}^+(\pi(a))} \left( |\nabla v|^2-\lambda_i^a p(x) |v|^2 \right)\dx : \ v\in H^1(D_{Ka_1}^+(\pi(a))), \ v=v_i^{ext} \text{ on } \partial D_{Ka_1}^+(\pi(a)) \right\}.
\]
We let $v_i=v_i^{int}$ in $D_{Ka_1}^+(\pi(a))$, $v_i=v_i^{ext}$ in $\Omega\setminus D_{Ka_1}(\pi(a))$. 
Notice that estimate \eqref{eq:v_i_int_estimates} holds in this case for every $1\leq i\leq k$ since $\varphi_k$ has no nodal line at $0$.
We take
\begin{align*}
F_{k} = \left\{ \Phi = \sum_{i=1}^{k} \alpha_{i} v_{i}: \, \alpha=(\alpha_1,\ldots,\alpha_k) \in \R^k \right\} \subset H^1_0(\Omega),
\end{align*}
so that
\begin{equation}\label{eq:lambda_max_M_N}
\lambda_k\leq \sup_{\Phi\in F_k} \frac{\|\nabla \Phi\|^2_{L^2(\Omega)}}{\int_{\Omega} p(x)|\Phi|^2 \mathrm{d}x}
= \sup_{\alpha\in\R^k} \frac{\alpha^T M \alpha}{\alpha^T N \alpha}
= \lambda_{max}(N^{-1}M).
\end{equation}
Here $\lambda_{max}(\cdot)$ is the largest eigenvalue of a matrix and $M,N$ are $k\times k$ matrices with entries
\[
\begin{split}
m_{ij}= \int_\Omega \nabla v_i\cdot\nabla v_j\dx 
=\lambda_i^a \int_{\Omega} p(x) v_i v_j\dx + \int_{\partial D_{Ka_1}^+(\pi(a))} \nabla(v_i^{int}-v_i^{ext})\cdot\nu v_j \dsigma, \\
n_{ij}=\int_\Omega p(x) v_i v_j \dx.
\end{split}
\]

Let us estimate $m_{k,k}$. We perform the following change of variables in order to work with the function $\psi_k^a$ defined in \eqref{eq:blowup}
\[
f_k^{a,ext}(y)=e^{-i\theta_e/2} \psi_k^a(y)
=\frac{v_k^{ext}(a_1 y + \pi(a))}{\sqrt{H(\varphi_k^a,Ka_1, \pi(a))}}, \quad
f_k^{a,int}(y)=\frac{v_k^{int}(a_1 y + \pi(a))}{\sqrt{H(\varphi_k^a,Ka_1, \pi(a))}}.
\]
By Theorem \ref{theorem:blow_up_almgren} we have that $f_k^{a,ext}\to e^{-i\theta_e/2} \psi_k$ in $H^1(D_K^+(0))$ as $a\to0$. Moreover, we have $f_k^{a,int}\to f_k^{int}$ in $H^1(D_K^+(0))$ as $a\to0$ and $-\Delta f_k^{int}=0$ in $D_K^+(0)$, $f_K^{int}=e^{-i\theta_e/2} \psi_k$ on $\partial D_K^+(0)$. From Proposition \ref{prop:unique_limit_profile} (ii), we deduce the following behavior of the harmonic extension $f_k^{int}$
\[
f_k^{int}(r,\theta)= C \left( 1 - \frac{ \beta}{\pi K^2} \right) r \cos\theta + C \sum_{\substack{n \geq 3 \\ n \text{ odd}}} \frac{b_n}{K^{2n}} \cos (n\theta) \, r^n, \quad r < K ,
\]
for $b_n$ as in Proposition \ref{prop:unique_limit_profile}, (ii), and $C$ given by \eqref{eq:constante_normalisation} with $K = k_\eps$.
Therefore we have
\[
\begin{split}
\int_{\partial D_{Ka_1}^+(\pi(a))} \nabla(v_k^{int}-v_k^{ext})\cdot\nu v_k \dsigma
=H(\varphi_k^a,Ka_1, \pi(a))\,  \int_{\partial D_K^+(0)} \nabla(f_k^{a,int}-f_k^{a,ext}) \cdot\nu f_k^{ext} \dsigma \\
= H(\varphi_k^a,Ka_1, \pi(a)) \, C^2 \, \left\{ -\beta +  O(K^{-2}) + o_{a_1}(1) \right\} = -C_k a_1^2
\end{split}
\]
for some $C_k>0$ as soon as $K$ is sufficiently large and $a_1$ is sufficiently small, where in the last step we used Lemma \ref{lemma:H_ka}.

We deduce that the matrices $M$ and $N$ appearing in \eqref{eq:lambda_max_M_N} have the following form
\begin{align*}
M =
\begin{pmatrix}
\lambda_{1}^{a}+O(a_1^{2})&         & O(a_1^{2})                       &  \\
                            &\ddots   &                                & O(a_1^{2}) \\
                  O(a_1^{2})  &         &\lambda_{k-1}^{a}+O(a_1^{2})&   \\
                            & O(a_1^{2})&                                &\lambda_{k}^{a}-C_{k}a_1^{2}+o(a_1^2)
\end{pmatrix}
\end{align*}
\begin{align*}
N =
\begin{pmatrix}
1 + O(a_1^{4}) &        &  O( a_1^{4})    &     \\
                 & \ddots &               &     \\
     O( a_1^{4})   &        &               &     \\
                 &        &  1 + O(a_1^{4}) &  
\end{pmatrix}.
\end{align*}
Since $\lambda_{k}$ is simple, proceeding similarly to Lemma \ref{lemma:greatesteigenvalue2}, we obtain 
\begin{align}\label{eq:lambda_k_less_lambda_k_a}
\lambda_{k} \leq \lambda_{k}^{a} - C_{k} a_1^{2} + o(a_1^2),
\end{align}
which concludes the proof.
Indeed, $N^{-1}M$ has the same form as $M$. When looking for the eigenvalues of this matrix we search the $t$ such that
\[
(\lambda_k^a-C_k a_1^2-t) Q_{k-1}(t,a_1^2)+a_1^4 Q_{k-2}(t,a_1^2)=0,
\]
where 
\[
Q_{k-1}(t,a_1^2)=\prod_{i=1}^{k-1} (\lambda_i^a+O(a_1^2)-t)
\]
and $Q_{k-2}(t,a_1^2)$ is a polynomial of degree $k-2$ in the variable $t$, which depends on $a_1$ with terms of order $O(a_1^2)$.
We set $\eps=a_1^2$ and we apply the implicit function theorem to
\[
f(\eps,x,t)=(x-t)Q_{k-1}(t,\eps)+\eps^2 Q_{k-2}(t,\eps)
\]
at the point $(0,\lambda_k^a-C_ka_1^2,\lambda_k^a-C_ka_1^2)$. We see that at any point $(0,\bar{x},\bar{x})$ we have
\[
\frac{\partial f}{\partial t}(0,\bar{x},\bar{x})=-Q_{k-1}(\bar{x},0)
=-\prod_{i=1}^{k-1} (\lambda_i-\bar{x}),
\]
so that the implicit function theorem applies at any point $(0,\bar{x},\bar{x})$ such that $\bar{x}\neq \lambda_i$ for every $i=1,\ldots,k-1$, and we have $t(\eps,\bar{x})=\bar{x}+o(\eps)$ in a neighborhood of $(\eps,x)=(0,\bar{x})$. Taking $\bar{x}=\lambda_k^a-C_ka_1^2$ we obtain \eqref{eq:lambda_k_less_lambda_k_a}.

\appendix

\section{Domains with conical singularities}\label{appendix:cone}

The proofs of the main theorems can be partially adapted to the case when $\Omega$ presents isolated conical singularities, as in the numerical simulations which appear in the Introduction. The results are qualitatively the same as for the smooth domain, but the rate of convergence of the eigenvalues depends on the aperture of the cone. We can interpret this fact in the following way: the zero boundary conditions on an acute angle of $\partial\Omega$ play the same role as the nodal lines of the eigenfunction. The tighter is the angle, the faster is the convergence.

Consider the following conical domain of aperture $\alpha$, for some $0<\alpha<\pi$
\begin{equation}\label{eq:cone}
\Omega=\left\{ (r,\theta): \, r\in (0,1), \, \theta\in \left(-\frac{\alpha}{2},\frac{\alpha}{2}\right) \right\}.
\end{equation}

The counterpart of Theorem \ref{theorem:at_least_one_nodal_line} holds.

\begin{theorem}\label{theorem:cone_at_least_one_nodal_line}
Let $\Omega$ be as in \eqref{eq:cone} and let $p$ satisfy \eqref{eq:weight_assumptions}.
Suppose that $\lambda_{k-1}<\lambda_{k}$ and that there exists an eigenfunction $\varphi_{k}$ associated to $\lambda_k$ having a zero of order $h/2\geq 2$ at the origin. 
Then there exists $C > 0$, not depending on $a$, such that
\begin{align*}
\lambda_{k}^{a} \leq \lambda_{k} - C |a|^{ h \frac{\pi}{\alpha} } \quad \text{for } 
a\to 0 \text{ along a nodal line of } \varphi_k.
\end{align*}
\end{theorem}

As for the analogous of Theorem \ref{theorem:no_nodal_lines}, we can prescribe the behavior of the eigenvalues only in case the pole approaches the vertex of the cone along the angle bisector. This restriction is related to the open problem presented in Remark \ref{rem:three_remarks} (i).

\begin{theorem}\label{theorem:cone_no_nodal_lines}
Let $\Omega$ be as in \eqref{eq:cone} and let $p$ satisfy \eqref{eq:weight_assumptions}.
Suppose that $\lambda_{k}$ is simple and that $\varphi_{k}$ has a zero of order $1$ at the origin. Then there exists $C > 0$, not depending on $a$, such that
\begin{align*}
\lambda_{k}^a \geq \lambda_{k} + C a_1^{ 2 \frac{\pi}{\alpha} } \quad \text{for } 
a=(a_1,0), \ a_1\to 0.
\end{align*}
\end{theorem}

The strategy of proof consists in applying the conformal map $x^\frac{\pi}{\alpha}$, so that the conical domain is transformed into the regular half ball $D_1^+(0)$. We end up with a singular equation of the following type
\begin{equation*}
(i\nabla+A_a)^2 \varphi_k^a = \lambda_k^a \left(\frac{\alpha}{\pi}\right)^2 
\frac{p(x)}{|x|^{2-\frac{2\alpha}{\pi}}} \varphi_k^a \quad \text{ in } D_1^+(0).
\end{equation*}
The singular potential $|x|^{-2+\frac{2\alpha}{\pi}}$ belongs to the Kato class, which allows to adapt the proofs of the previous sections. In particular, the following Hardy inequality holds:  for every $\eps>0$ there exists a positive constant $C$ such that
\begin{equation}\label{eq:hardy_generalized}
\frac{C}{r^\eps} \int_{D_r(0)} \frac{|u|^2}{|x|^{2-\eps}} \dx
\leq \int_{D_r(0)} |\nabla |u||^2 \dx + \frac{1}{r}\int_{\partial D_r(0)} |u|^2 \dsigma,
\end{equation}
for every $u\in H^1(D_r(0),\C)$ and for every $r>0$ (see \cite{WangZhu2003}). By combining with the diamagnetic inequality
\[
\int_{D_r(0)} |\nabla |u||^2 \dx \leq \int_{D_r(0)} |(i\nabla +A_a) u |^2 \dx,
\]
we obtain the counterpart of the Poincaré inequality \eqref{eq:Poincaré}.

Concerning Proposition \ref{proposition:asymptotic_laplacian_eigenfunctions}, its validity in case of a singular potential belonging to the Kato class is stated in \cite[Theorem 1.3]{FFT2011}.

\section{Green's function for a perturbation of the Laplacian}\label{appendx:green}

\begin{lemma}\label{lemma:appendix_green}
Consider the set of equations (depending on the parameter $\eps$) $-\Delta f=\eps c(x)f$ in $\Omega\subset\R^2$ bounded, with $c\in L^\infty(\Omega)$. For $\eps$ sufficiently small there exists a Green's function $G(x,y)$ such that the following representation formula holds for $x\in\Omega$
\[
f(x)=-\int_{\partial\Omega} f\partial_{\nu}G(x,\cdot)\dsigma(y).
\]
Moreover, for every $1\leq p<\infty$ there exists $C$ independent from $\eps$ such that we have
\[
\|\partial_{x_i}G(x,\cdot)-\partial_{x_i}\Phi(x,\cdot)\|_{W^{1,p}(\partial\Omega)}\leq C\eps,
\]
for $x\in\Omega$, where $\Phi$ is the Green function of the Laplacian with homogeneous Dirichlet boundary conditions in $\Omega$.
\end{lemma}
\begin{proof}
We define the Green function as $G(x,y)=\Gamma(y-x)+L(x,y)$, where $\Gamma(x)=-\frac{1}{2\pi}\log|x|$ is the fundamental solution of the Laplacian in $\R^2$ and $L(x,\cdot)$ solves, for $x\in\Omega$,
\[
\left\{\begin{array}{ll}
-\Delta L(x,y)-\eps c(y) L(x,y)=\eps c(y) \Gamma(y-x) &\quad y\in \Omega \\
L(x,y)=-\Gamma(y-x)&\quad y\in \partial\Omega.
\end{array}\right.
\]
Notice that this equation admits a solution for $\eps$ small because the quadratic form
\begin{equation}\label{eq:quadratic_form_positive}
\int_\Omega(|\nabla v|^2-\eps c(x)v^2)\dx
\end{equation}
is coercive for $v\in H^1_0(\Omega)$, and moreover $\Gamma \in L^2(\Omega)$.

The validity of the representation formula is standard. Indeed, the following identity holds (see for example \cite{EvansBook}, equation (25) in paragraph 2.2.4)
\[
f(x)=-\int_\Omega \Gamma(y-x)\Delta f(y)\,\textrm{d}y +\int_{\partial\Omega}(\Gamma(y-x)\partial_{\nu}f(y)- f(y)\partial_{\nu}\Gamma(y-x))\,\textrm{d}\sigma(y).
\]
By using the Green formula
\[
\int_{\Omega} (\Delta L \,f-L\Delta f)\,\textrm{d}y =
\int_{\partial\Omega}(\partial_{\nu}L\, f-L\partial_{\nu}\,f)\,\textrm{d}\sigma(y)
\]
and the equation satisfied by $L(x,\cdot)$, we obtain the representation formula for $f$.

In order to estimate $\partial_{x_i}(G-\Phi)$, we write $\Phi(x,y)=\Gamma(y-x)+H(x,y)$, with
\[
\left\{\begin{array}{ll}
-\Delta H(x,y)=0 &\quad y\in \Omega \\
H(x,y)=-\Gamma(y-x)&\quad y\in \partial\Omega,
\end{array}\right.
\]
so that $\partial_{x_i}(G-\Phi)=\partial_{x_i}(L-H)=:u$ solves
\[
\left\{\begin{array}{ll}
-\Delta u-\eps c(y) u=\eps c(y) \partial_{x_i}\Phi(x,\cdot) &\quad y\in \Omega \\
u=0&\quad y\in \partial\Omega.
\end{array}\right.
\]
We apply Poincar\'e inequality and the positivity of the quadratic form in \eqref{eq:quadratic_form_positive} as follows
\[
\begin{split}
\|u\|_{H^1(\Omega)} & \leq C \|\nabla u\|_{L^2(\Omega)} \\
& \leq C \left( \int_\Omega(|\nabla u|^2-\eps c(y)u^2)\,\textrm{d}y \right)^{1/2} \\
&= C \left( \int_\Omega \eps c(y) \partial_{x_i}\Phi(x,y) u\,\textrm{d}y \right)^{1/2}.
\end{split}
\]
Since $\partial_{x_i}\Phi(x,\cdot)\in L^q(\Omega)$ for $1\leq q<2$, we can apply the H\"older inequality and the Sobolev embedding to obtain
\[
\|u\|_{H^1(\Omega)} \leq C\eps^{1/2} \left( \|\partial_{x_i}\Phi\|_{L^{3/2}(\Omega)} \|u\|_{L^3(\Omega)} \right)^{1/2} 
\leq C \eps^{1/2} \|u\|_{H^1(\Omega)}^{1/2}. 
\]
Finally, using again the Sobolev embeddings and a bootstrap argument, we obtain that $u \in W^{2,q}(\Omega)$ for every $1 \leq q < 2$ and
\begin{align} \label{eq:estimate_W_2_q}
\| u \|_{W^{2,q}(\Omega)} \leq C \varepsilon.
\end{align}

\end{proof}

%\nocite{*}
%\bibliographystyle{plain}

%\bibliographystyle{abbrv}
%\bibliography{bibliography.bib}

\end{document}